%% file: ArxivMainCapLDPMF.tex
\documentclass[11pt,a4paper,reqno]{amsart}
\usepackage{color,latexsym,amsfonts,amssymb,bbm,comment}
\usepackage{hyperref}
\usepackage{amsmath,cite}
\usepackage{amsthm}
\usepackage{fullpage}
\usepackage{graphicx} 
\usepackage{tikz} 
\usepackage{caption,subcaption}

\usepackage{pgfplots}

\newcommand{\E}{\mathbb{E}} 

\newcommand{\R}{\mathbb{R}}

\newcommand\supp{\mathop{\rm{ supp}}}

\def\eq{\begin{equation}}
\def\en{\end{equation}}

\newtheorem{theorem}{Theorem}[section]
\newtheorem{corollary}[theorem]{Corollary}
\newtheorem{lemma}[theorem]{Lemma}
\newtheorem{proposition}[theorem]{Proposition}
\theoremstyle{definition}

\def\DD{{\Bbb L}}

\def\e{{\varepsilon}}

\def\D{\Delta}
\def\a{\alpha}
\def\sm{\setminus}

\def\b{\beta}
\def\MM{{\cal M}}

\def\TT{{\cal T}}
\def\TT{{\mc B}}
\def\e{\varepsilon}

\def\phi{\varphi}
\def\g{\gamma}
\def\la{\lambda}
\def\k{\kappa}
\def\r{\rho}
\def\de{\delta}

\def\D{\Delta}
\def\L{\Lambda}
\def\G{\Gamma}

\def\P{{\Phi}}

\def\T{\T}

\def\MM{\mathcal M}

\def\TT{{\mathcal T}}

\def\bb{{\mathfrak{b}}}
\def\nn{{\mathfrak{n}}}

\def\DD{{\mathbb D}}
\def\LL{{\mc L}}

\def\V|{{\Vert}}
\def\LLL{{\mathfrak L}}

\def\bb{{\bf{b}}}

\def\d{{\rm d}}
\def\E{\mathbb{E}}
\def\I{\mathcal{I}}

\def\one{\mathbbmss{1}}
\def\mc{\mathcal}

\def\ms{\mathsf}

\def\one{\mathbbmss{1}}
\def\P{\mathbb{P}}
\def\R{\mathbb{R}}

\def\TF{t_{\ms f}}
\def\XTi{X_i}

\def\bdde{\b^{\de}}
\def\bddde{\b^{\uparrow,{\de}}}
\def\gdde{\g^{\de}}

\def\Lla{L_{\la}}

\def\muR{\mu_{\ms{R}}}
\def\nuR{\nu_{\ms{R}}}
\def\lla{l_\la}
\def\Zla{Z^\la}
\def\Yla{Y^\la}
\def\mula{\mu_\la}

\def\fla{f_\la}
\def\gla{g_\la}
\def\Zdla{Z^{\la,\de}}

\def\itf{[0,\TF]}
\def\itfW{\itf \times W}
\def\itfWY{\itfW \times Y^\la}
\def\itfiww{\itf \times [0,1] \times W^2}
\def\itW{[0,t] \times W}

\def\Ia{I_\alpha}
\def\nns{\nn^*}
\def\nnsp{\nn^{*,+}}
\def\nnsm{\nn^{*,-}}

\def\Zlap{Z^{\la,+}}
\def\Zlam{Z^{\la,-}}
\def\Zlas{Z^{\la,*}}
\def\MMs{\MM^*}
\def\mus{\mu^*}
\def\mulim{\musp \otimes \muR}
\def\musp{\mu^{\ms{s}}_{\ms{T}}}
\def\must{\mu^{\ms{t}}_{\ms{T}}}

\def\LLLd{\LLL^\de_\la}

\def\rla{r_\la}

\keywords{large deviations, measure-valued differential equations, entropy, capacity, relay}
\subjclass[2010]{Primary 60F10; secondary 60K35}

\begin{document}
\author{Christian Hirsch}
\author{Benedikt Jahnel}
\author{Robert Patterson}
\thanks{Weierstrass Institute Berlin, Mohrenstr. 39, 10117 Berlin, Germany; E-mail: {\tt christian.hirsch@wias-berlin.de}, {\tt benedikt.jahnel@wias-berlin.de}, {\tt robert.patterson@wias-berlin.de}.}

\title{Space-time large deviations in capacity-constrained relay networks}

\date{\today}

\begin{abstract}
We consider a single-cell network of random transmitters and fixed relays in a bounded domain of Euclidean space. The transmitters arrive over time and select one relay according to a spatially inhomogeneous preference kernel. Once a transmitter is  connected to a relay, the connection remains and the relay is occupied. If an occupied relay is selected by another transmitters with later arrival time, this transmitter becomes frustrated. We derive a large deviation principle for the space-time evolution of frustrated transmitters in the high-density regime.
\end{abstract}

\maketitle 



\input{introMF.tex}

\input{ProofTheorem1}

%
\input{SupportTheorem1}

\input{ProofTheorem2}

\input{SupportTheorem2}

\section*{Acknowledgments}
This research was supported by the Leibniz program Probabilistic Methods for Mobile Ad-Hoc Networks. The authors thank P.~Keeler, W.~K\"onig and M.~Renger for interesting discussions and comments.

\bibliography{../../wias}
\bibliographystyle{abbrv}

\end{document}

%% file: introMF.tex
\section{Introduction and main results}
We consider a single-cell communication network of random \textit{transmitters} and fixed \textit{relays}. Every transmitter tries to send data to a central entity via one relay according to a spatially dependent preference function. Each relay can only serve one transmitter and the transmitters are competing for this shared capacity. In particular, a group of transmitters might not be successful in finding relays to release their data to and therefore become \textit{frustrated}. We will assume that the start of the data transmission is time dependent so that the set of frustrated transmitters gradually increases over time. We present a large deviation principle (LDP) for the measure-valued process of frustrated transmitters 
%
when their number increases. 

\medskip
The motivation for this work is to derive the LDP  for a capacity-constrained network embedded in the Euclidean space. So far in the literature, there have been two separate approaches. On the one hand, considerable work has been done to understand the large deviation behavior of sophisticated capacity-constrained networks in a mean-field setting, see for example~\cite{gramMel1,gramMel2}. On the other hand, driven by recent developments in wireless networks, from an engineering perspective, there has been a surge in research activities to develop a fundamental understanding of spatial effects in models that are based on stochastic geometry~\cite{caireDynamic,baccelliDynamic}. In the present paper, we analyze a simple model of a spatial relay network which is to be seen as a first step into the realm of space-time LDPs for capacity-constrained networks.

\medskip
More specifically, consider fixed relays at locations $Y^\la=(y_j)_{1\le j\le n_\la}$ in a compact window $W\subset\R^d$ with boundaries of vanishing Lebesgue measure. We investigate the high-intensity regime and thus assume that the empirical distribution 
$$l_{\la}=\la^{-1}\sum_{y_j\in Y^\la}\de_{y_j}$$
converges weakly, as $\la\uparrow\infty$, to some probability measure $\mu_{\ms R}$ on $W$. 
Further there will be transmitters distributed according to a Poisson point process $X^\la$ in $W$. Its intensity measure is of the form $\la\mu^{\ms s}_{\ms{T}}$ with $\la>0$ and $\mu^{\ms s}_{\ms T}$ a finite measure on $W$ which is absolutely continuous w.r.t.~the Lebesgue measure. Initially, all relays are idle and the transmitters do not send data. For each transmitter $\XTi$ there is a randomly distributed time $T_i\in[0,\TF]$ and in $(T_i,\TF]$ there will be constant data transmission. The times $T_i$ are assumed to be iid with distribution $\mu^{\ms t}_{\ms T}$ which is absolutely continuous w.r.t.~the Lebesgue measure on $[0,\TF]$. 

The transmitters are assumed to have basic knowledge about the transmission quality to each of the relays. More precisely, each transmitter $\XTi$ uses this information to assign a spatial preference $\k(\XTi,y_{j})\in[0,1]$ for the connection from $\XTi$ to $y_{j}$. 
At time $T_i$ the user $\XTi$ tries to send its data to a relay $y_{j}$ chosen according to the \textit{preference kernel}
\begin{align}\label{Kappa}
\k(y_{j}| \XTi)=\frac{\k(\XTi, y_{j})}{\sum_{y_{k}\in Y}\k(\XTi, y_{k})},
\end{align}
so that the selection probability is proportional to the spatial preference function $\k$. In our model, data transmission fails if the chosen relay is already occupied by some other transmitter that has established a connection earlier in time. 

\medskip
Since every transmitter $\XTi$ keeps the connection until time $t_{\ms f}$, the time-dependent status of its target relay $Y^{(i)}(t)\in\{0,1\}$ is a step function, starting in $0$ as being idle and jumping to $1$ at the time where it becomes busy. In particular $\G^\la=\{\G^\la_t\}_{0\le t\le\TF}$ with 
\begin{align}\label{Busy_Process}
	\G^\la_t=\frac{1}{\la}\sum_{\XTi\in X}\one\{{\XTi(t)=1}\}\de_{\XTi}
\end{align}
denotes the normalized, time-dependent random measure of frustrated transmitters.
Here $X_i$ is the (time invariant) position of a relay, but $\XTi(t)$ represents the time-dependent status of frustration of the transmitter $\XTi$. More precisely, $\XTi(t)$ equals $0$ for $t<T_i$ and jumps to $1$ at $t=T_i$ if the chosen relay is already occupied, i.e. $Y^{(i)}(t-)=1$. 
Figure~\ref{Fig} provides a snapshot of the relay network after a finite time.

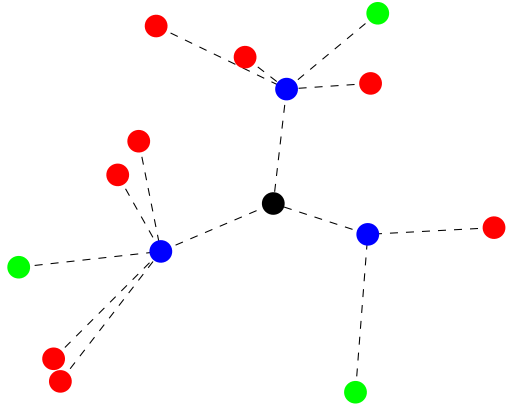
\begin{figure}[!htpb]
\centering
\input{smile.tex}
\caption{Collection of transmitters (green and red) communicating with one central entity (black) via one relay (blue) each. Red transmitters are frustrated due to low capacity at the associated relay. Green transmitters are satisfied as they successfully send data to a relay.}
\label{Fig}
\end{figure}

\medskip
From the perspective of the network operator, it is critical to understand the time- and space-dependent process of frustrated transmitters in large networks. Once this understanding is achieved, it is possible to answer questions like:
\begin{center}
	\emph{What is the overall proportion of transmitters which are frustrated at a given point in time?}
	\emph{	Are most of the frustrated transmitters located in a specific area?}
\end{center}
 Our main result provides a probabilistic description of the process of frustrated transmitters in an asymptotic regime where the number of devices tends to infinity. In particular, we perform a large deviation analysis for the empirical measure of frustrated transmitters $\G^\la$, where rates of convergence for unlikely events are derived.

\subsection{Large deviations without preference}
Note that $\G^\la$ allows us to keep track of transmitter locations but not on the chosen connections to specific relays. 
In particular, for general preference kernels, $\G^\la$ is not Markovian since the required spatial information of occupied relays for a new transmitter request at time $t$ cannot be extracted from $\G^\la_{t-}$. However, for $\k\equiv 1$, the Markovianity of $\G^\la$ can be preserved since transmitters have no spatial preference in their choice of relays. 
Therefore, we establish the case $\k\equiv 1$ first and use it as a basis for the general case.

For this, we first note that
the process of frustrated transmitters $\G^\la$ can be recovered from the process of {\it{satisfied transmitters}} $B^\la=\{B^\la_t\}_{0\le t\le\TF}$. This process is easier to describe as it coincides with the number of busy relays.
More precisely, when at time $t=T_i$ a transmitter request from $X_i$ arrives, the chosen relay is already busy with probability given by the total proportion of idle relays and the transmitter becomes frustrated. In this case, $B^\la$ stays constant at $t$. Otherwise, if the chosen relay is idle, the relay becomes busy and $B^\la$ grows by $\la^{-1}\de_{X_i}$ at $t$. As the number of satisfied transmitters equals the number of busy relays, this has probability $1-B^\la_{t-}(W)/r_{\la}$ where $r_{\la}=|Y^\la|/\la$. 
Note that this random choice of relays can be encoded by assigning a uniform random variable $U_i\in[0,1]$ to transmitter $\XTi$. If $U_i\in[B^\la_{t-}(W)/r_{\la},1]$ then $\XTi$ connects to an idle relay. The encoding of the spatial relay configuration into a $[0,1]$-valued threshold is illustrated in Figure~\ref{flatFig}.

\begin{figure}[!htpb]
\centering
\input{flatFig}
\caption{A transmitter (black) chooses a relay at random (left), where the relay can either be already busy (red) or idle (green). Without spatial preferences the relay information can be reduced to a single threshold in $[0,1]$ (right).}
\label{flatFig}
\end{figure}
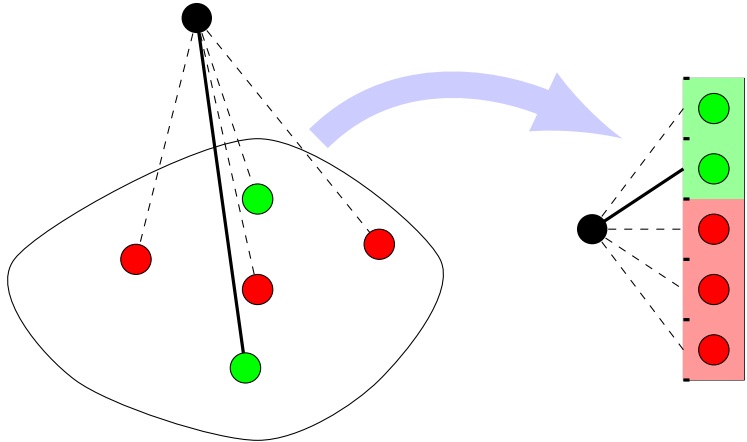

More precisely, the evolution of the process of satisfied transmitters $B^\la$ is given as the solution of the time integral equation
\begin{align}\label{DE_Discrete}
B_t(\d x)=\int_0^t\int_{B_{s-}(W)/r_{\la}}^1L_\la(\d s,\d u,\d x)=\int_0^tL_\la(\d s,[B_{s-}(W)/r_{\la},1],\d x),
\end{align}
where
$$L_\la=\la^{-1}\sum_{\XTi\in X^\la}\de_{(T_i,U_i,X_i)}$$ 
is the empirical measure of transmitters. In particular, the Poisson point process $X^\la$ of transmitters carries marks for data request time and choice variable. Its intensity measure $\la\mu_{\ms T}$ in $V=[0,\TF]\times[0,1]\times W$ is given by
$$\mu_{\ms T}=\mu^{\ms t}_{\ms T}\otimes U\otimes\mu^{\ms s}_{\ms T}$$ 
and $U$ is the uniform distribution on $[0,1]$. 

In words, equation~\eqref{DE_Discrete} describes the evolution of the empirical measure of satisfied transmitters.
This measure gains mass $\la^{-1}$ at position $x$ if there is an additional transmitter request from position $x$ at time $t$ and this transmitter picks a relay that is idle at time $t-$.  From $B^\la$ the random measure of frustrated transmitters can be recovered via
\begin{align*}
\G^\la_t(\d x)=L_\la([0,t],[0,1],\d x)-B^\la_t(\d x).
\end{align*}

\medskip
Note that, as the number of devices tends to infinity, in the limit, the point masses associated to individual devices disappear and the picture becomes continuous. Thus, we need to introduce processes of frustrated transmitters for absolutely continuous measures. More precisely, let $\nu$ be a finite measure in $\MM=\MM(V)$
where $\nu(\d t,\d u,\d x)$ has the interpretation of the intensity of transmitters at $\d x$ with data-entry time $\d t$ and choice variable $\d u$. 
Then, for general measures and normalized number of relays $r$, equation \eqref{DE_Discrete} has the form
\begin{align}\label{DE_Cont_r}
\b_t(\d x)=\int_0^t\int_{\b_{s-}(W)/r}^1\nu(\d s,\d u,\d x)=\int_0^t\nu(\d s,[\b_{s-}(W)/r,1],\d x).
\end{align}
For illustration purposes let us present two examples here. 
\begin{enumerate}
\item For the empirical measure $\nu=L_\la$ of the Poisson point processes $X^\la$ and $r=r_\la$, the unique solution $\b$ for \eqref{DE_Cont_r} is given by $B^\la$.
\item For the a priori measure $\mu_{\ms T}$ as a driver and normalized relay number $r=1$, the unique solution of \eqref{DE_Cont} is given by $\b\mu^{\ms s}_{\ms T}$ where
$$\b_t=\mu^{\ms s}_{\ms{T}}(W)^{-1}(1-e^{-\mu^{\ms s}_{\ms{T}}(W)\mu^{\ms t}_{\ms{T}}([0,t])}).$$
\end{enumerate}

\medskip
Note that, instead of \eqref{DE_Cont_r} with $r=1$, it suffices to consider the scalar equation
\begin{align}\label{DE_Cont}
\b_t=\int_0^t\nu(\d s,[\b_{s-},1],W).
\end{align}
Using Schauder's fixed point theorem and monotonicity, we will show in Subsection~\ref{Existence and uniqueness of solutions} that existence and uniqueness of solutions of \eqref{DE_Cont} for 
absolutely continuous measures can be established.
With the steps we have just described we arrive at a solution which we will denote $\b(\nu)$.
Moreover, in the absolutely-continuous case, the solution will be continuous and increasing in time.

From the solution $\b(\nu)$ one can compute the normalized process of frustrated transmitters $\g(\nu)$ via the formula
\begin{align}
\label{frustUserDef}
\g_t(\nu)(\d x)=\nu([0,t],[0,1],\d x)-\b_t(\nu)(\d x).
\end{align}

\medskip
In the following theorem we show the LDP  for $\G^\la$ in the setting where $\k\equiv 1$. 
Recall the definition of the relative entropy
$$h(\nu|\mu)=\int\log\frac{\d\nu}{\d\mu}\d\nu-\nu(V)+\mu(V)$$
if $\nu\ll\mu$ with $h(\nu|\mu)=\infty$ otherwise.
We consider $\G^\la$ as a measure-valued process and work in the Skorohod space. That is, we consider
$$\DD=\{f\in \MM(W)^{[0,\TF]}:\, f \text{ is c\`adl\`ag w.r.t.~the weak topology on }\MM(W)\}$$
equipped with the Skorohod topology, for details see for example \cite{FeKu06}.

\medskip
\begin{theorem}\label{LDP_NoSpatial}
The family of measure-valued processes $\G^\la$ satisfies the LDP  in $\DD$ with good rate function given by $I(\g)=\inf_{\nu\in\MM:\, \g(\nu)=\g}h(\nu|\mu_{\ms T})$.
\end{theorem}
Note that a scalar variant of Theorem~\ref{LDP_NoSpatial} appears in~\cite[Theorem 2.7]{binBalls} when transmitters are interpreted as bins and relays as balls. This provides an application of this classical model from random discrete structures to communication networks. 
However, the results in~\cite{binBalls} cannot be used to prove Theorem~\ref{LDP_NoSpatial}, since the balls arrive at deterministic times, whereas in our setting also the arrival times of the transmitters are random and not necessarily homogeneous. 

\medskip
Moreover, in the scalar setting, it is possible to explicitly perform the minimization 
only over the choice component. Then, the rate function of Theorem~\ref{LDP_NoSpatial} can be expressed as 
\begin{align*}
I(\g)=\inf_{\b}\int_0^{\TF} \big[h(\dot\b_t|1-\b_t)+h(\dot\g_t|\b_t)\big]\d t.
\end{align*}
Here, we assumed $\mu^{\ms t}_{\ms T}(\d t)=\one_{[0,\TF]}(t)\d t$ for simplicity and the infimum is taken over increasing and absolutely continuous paths. This form of the rate function coincides with the one derived in \cite{ShWe95} in the setting of chemical-reaction networks. 
The interpretation of our communication-network evolution is then that the species of frustrated and satisfied transmitters are generated at rates $\b_t$ respectively $1-\b_t$. Theorem~\ref{LDP_NoSpatial} is not covered by the standard results in~\cite{Leon95, ShWe95}, since in our setting, these rates are not bounded away from zero.

In the literature attempts have been made to relax the assumption of strictly positive rates~\cite{dimRates}. However, our model is also not covered by the results in~\cite{dimRates}, as the crucial \emph{interior cone property} is violated: Once all relays are occupied it is not possible to move back to a state where satisfied users are generated at positive rate.

Finally, it is difficult to extend the interpretation of our network as a chemical reaction if we take spatial resolution into account. Then, the space of species would become uncountable since spatial locations have to be tracked. This is another reason for our decision to rely on marked Poisson point processes and measure-valued differential equations in Theorem~\ref{LDP_NoSpatial}.

\medskip
Let us also point out that, instead of the Skorohod topology, other topologies for process LDPs have been considered in the setting of chemical-reaction networks. As will be apparent from the proof, Theorem~\ref{LDP_NoSpatial} can be extended, for example, to the bounded-variation topology as considered for example in \cite{PaRe16}.

\subsection{Large deviations with preference}
In this section, we deal with spatial preferences of transmitters. As a consequence, the probability to send to a certain location depends on the spatial location of transmitters and relays and not just on the number of relays in the entire domain. In particular, the encoding into a single $[0,1]$-valued threshold falls short of capturing the information required for describing the evolution of frustrated transmitters. However, for a sufficiently smooth preference kernel, for a given transmitter,  the relay choice is approximately uniform in a neighborhood around any given relay location. Therefore, as an approximation we partition $W$ into a finite number of patches. 
Then to each of these patches we associate a separate $[0,1]$-valued threshold describing the approximate proportion of busy relays in that patch. 
This encoding is illustrated in Figure~\ref{prefFig}.
\begin{figure}[!htpb]
\centering
\input{prefFig}
\caption{A transmitter (black) chooses a relay location at a coarse scale according to a spatial preference function (left). At a fine scale the configuration of busy relays can be encoded in a $[0,1]$-valued threshold as in the setting of flat preference kernels (right).}
\label{prefFig}
\end{figure}
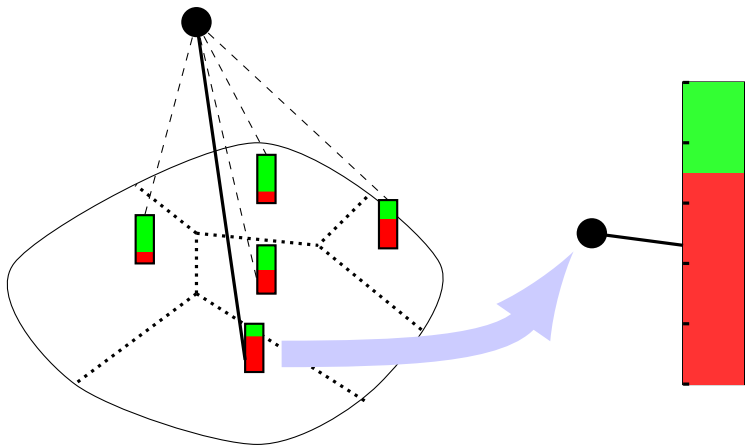

Here, we assume $\mu_{\ms R}$ to be absolutely continuous w.r.t.~the Lebesgue measure.
 The spatial preference function may be discontinuous and non-positive, but for sufficiently high densities from any transmitter location it should be possible to connect to some relay location.
More precisely, we assume that $\k$ is jointly continuous $\mu^{\ms s}_{\ms T}\otimes\mu_{\ms R}$-almost everywhere and 
for all $x\in W$ there exists $y\in W$ such that $\k(x,y)>0$, $y\in\supp(\muR)$ and $(x,y)$ is a continuity point of $\k$.
Recall that a transmitter at location $x\in W$ chooses a relay at location $\d y$ with probability 
\begin{align}\label{kappa}
\k_{l_\la}(\d y|x)=
\frac{\k(x,y)}{\int \k(x,z)l_\la(\d z)}l_\la(\d y)
\end{align}
where, by our assumption, for sufficiently large $\la$ the denominator is bounded away from 0 uniformly in $x$.

\medskip
Next we consider the interplay between the spatial location of the relays and how the spatial preferences of the transmitters evolve in time. Note that the process of transmitter requests to a relay at location $\d y$ is a Poisson point process $Z^\la$ on $V'=V\times W$ with intensity measure
\begin{align*}
	\mu(l_\la)(\d t, \d u,\d x,\d y) = \k_{l_\la}(\d y|x)\mu_{\ms T}(\d t, \d u, \d x).
\end{align*}

As in Theorem~\ref{LDP_NoSpatial} the LDP will be obtained by making use of the observation that the measure of frustrated transmitters $\G^\la$ can be interpreted as a functional of the empirical measure $\mathfrak \LLL_\la\in \MM'=\MM(V') $ associated to $Z^\la$.
In particular, in the large deviations it is possible for the process to distort the new a priori measure $\mu(\mu_{\ms R})$ into an absolutely-continuous transmitter-request distribution $\nn$. More precisely, we generalize~\eqref{DE_Cont} into the measure-valued time integral equation
\begin{align}\label{DE_Cont_Gen}
\bb_t(\d x,\d y)=\int_0^t\int_{\tfrac{\d\bb_{s-}(W,\cdot)}{\d\nu_{{\ms{R}}}}(y)}^1\nn(\d s,\d u,\d x,\d y)=\int_0^t\nn(\d s,[\tfrac{\d\bb_{s-}(W,\cdot)}{\d\nu_{{\ms{R}}}}(y),1],\d x, \d y)
\end{align}
where we allow the relay measure $\nu_{\ms R}$ also to be general in order to cover both, empirical measures as well as absolutely-continuous measures. For example,
\begin{enumerate}
\item if the driving measure is given by $(\LLL_\la ,l_\la)$, then~\eqref{DE_Cont_Gen} has a unique solution $\bb(\d x,\d y)$ and, for $\k\equiv 1$, $B^\la(\d x)=\bb(\d x,W)$ in distribution.
\item 
if $\nn=\mu(\muR)$, then a solution is given by
\begin{align*}
\bb_t(\d x,\d y)=(1-e^{-\mu^{\ms t}_{\ms{T}}([0,t])\int\k(y|z)\mu^{\ms s}_{\ms{T}}(\d z)})\frac{\k(y|x)\mu^{\ms s}_{\ms T}(\d x)}{\int\k(y|z)\mu^{\ms s}_{\ms{T}}(\d z)}\muR(\d y).
\end{align*}
\end{enumerate}
Note that, for general driving measures $(\nn,\nu_{\ms R})$, existence of solutions for equation~\eqref{DE_Cont_Gen} is unclear. However, existence of solutions can be established if we assume $\nn$ to be of the form
$$\nn(\d t,\d u,\d x,\d y)=\nn_y(\d t,\d u,\d x)\nu_{\ms R}(\d y)$$
where $(\nn_y)_{y\in W}$ is a transition kernel $W\to\MM$ such that every $\nn_y(\d t,\d u,\d x)$ is absolutely continuous w.r.t.~$\mu_{\ms{T}}$. 
Indeed, then a solution of \eqref{DE_Cont_Gen} is given by  
\begin{align}\label{MidSolution}
\bb(\d x,\d y)=\b(\nn_y)(\d x)\nu_{\ms R}(\d y)
\end{align}
and will be denoted by $\bb(\nn,\nu_{\ms R})$.

\medskip
If $\bb(\nn,\nu_{\ms R})$ is well-defined, we can compute the process of frustrated transmitters $\g{(\nn,\nu_{\ms R})}$ via the formula
\begin{align*}
\g_t{(\nn,\nu_{\ms R})}(\d x)=\nn([0,t],[0,1],\d x, W)-\bb_t(\nn,\nu_{\ms R})(\d x, W)
\end{align*}
and in particular $\G^\la=\g(\LLL_\la ,l_\la)$ equals \eqref{Busy_Process} in distribution. 
Now, we present the main result of this paper.

\begin{theorem}\label{LDP_Spatial}
The family of measure-valued processes $\G^\la$ satisfies the LDP in $\DD$ with good rate function given by $I(\g)=\inf_{\nn\in \MM':\, \g(\nn,\mu_{\ms R})=\g}h(\nn|\mu(\mu_{\ms R}))$.
\end{theorem}

Note that at first sight, $\G^\la$ is an object that requires knowledge on how many transmitters choose a relay and not just the number of transmitters in a given area that choose relays in another given area. Therefore, it may come as a surprise that we can state Theorem~\ref{LDP_Spatial} as a measure-valued LDP and not as a LDP on the level of spatial configurations. To reconcile Theorem~\ref{LDP_Spatial} with the reader's intuition, we note that after approximation by flat preference functions, we deal with an independent collection of processes of the type considered in Theorem~\ref{LDP_NoSpatial}. This allows us to aggregate the information about an entire local configuration of occupied relays into a single number.

\medskip
The rest of the manuscript is organized as follows. Sections~\ref{thm1Sec} and~\ref{thm2Sec} provide high-level overviews for the proofs of Theorems~\ref{LDP_NoSpatial} and~\ref{LDP_Spatial}, respectively. Detailed proofs for all supporting results can be found in Sections~\ref{thm1SupSec} and~\ref{thm2SupSec}.

%% file: smile.tex
\begin{tikzpicture}[xscale=0.025,yscale=0.025,rotate=65]

\draw[dashed] (307.9873,269.3195)--(250,250);
\draw[dashed] (364.7403,242.8243)--(307.9873,269.3195);
\draw[dashed] (309.3138,345.6901)--(307.9873,269.3195);
\draw[dashed] (314.1278,296.2691)--(307.9873,269.3195);
\draw[dashed] (329.4118,230.5234)--(307.9873,269.3195);
\draw[dashed] (201.9386,292.8992)--(250,250);
\draw[dashed] (162.7232,357.1526)--(201.9386,292.8992);
\draw[dashed] (256.1847,198.0011)--(250,250);
\draw[dashed] (287.4920,139.2742)--(256.1847,198.0011);
\draw[dashed] (177.5109,168.4481)--(256.1847,198.0011);
\draw[dashed] (229.1610,330.6283)--(201.9386,292.8992);
\draw[dashed] (126.4416,319.9642)--(201.9386,292.8992);
\draw[dashed] (117.0718,311.6700)--(201.9386,292.8992);
\draw[dashed] (250.0329,328.1565)--(201.9386,292.8992);

\fill[color=black] (250,250) circle(6);

\fill[color=blue] (307.9873,269.3195) circle(6);
\fill[color=green] (364.7403,242.8243) circle(6);
\fill[color=red] (309.3138,345.6901) circle(6);
\fill[color=red] (314.1278,296.2691) circle(6);
\fill[color=red] (329.4118,230.5234) circle(6);

\fill[color=red] (250.0329,328.1565) circle(6);
\fill[color=red] (117.0718,311.6700) circle(6);
\fill[color=red] (126.4416,319.9642) circle(6);

\fill[color=red] (229.1610,330.6283) circle(6);
\fill[color=blue] (256.1847,198.0011) circle(6);
\fill[color=red] (287.4920,139.2742) circle(6);
\fill[color=green] (177.5109,168.4481) circle(6);
\fill[color=blue] (201.9386,292.8992) circle(6);
\fill[color=green] (162.7232,357.1526) circle(6);

\end{tikzpicture}

%% file: flatFig.tex
\begin{tikzpicture}[scale = 0.8]

	\draw plot [smooth cycle] coordinates {(0,0) (3,-1) (5,0) (6,2) (3,4) (-1,2)};

\fill[color=black] (2,6) circle(0.25);

	\draw[dashed] (2,6)--(1,2) ;
\draw[dashed] (2,6)--(5,2.25) ;
\draw[very thick] (2,6)--(2.8,0.2) ;
\draw[dashed] (2,6)--(3,1.5) ;
\draw[dashed] (2,6)--(3,3) ;

\fill[color=red] (1,2) circle(0.25);
\fill[color=red] (5,2.25) circle(0.25);
\fill[color=green] (2.8,0.2) circle(0.25);
\fill[color=red] (3,1.5) circle(0.25);
\fill[color=green] (3,3) circle(0.25);

\draw (1,2) circle(0.25);
\draw (5,2.25) circle(0.25);
\draw (2.8,0.2) circle(0.25);
\draw (3,1.5) circle(0.25);
\draw (3,3) circle(0.25);

	\draw[->, >=latex, blue!20!white, line width = 10pt, bend  left] (4,4) to[in =155, out = 45] (9,4);

	\draw (10,0)--(11,0)--(11,5)--(10,5)--cycle;

\fill[color=black] (8.5,2.5) circle(0.25);

	\fill[color=green!40!white] (10,3)--(10,5)--(11,5)--(11,3)--cycle;
	\fill[color=red!40!white] (10,0)--(10,3)--(11,3)--(11,0)--cycle;

\fill[color=red] (10.5,0.5) circle(0.25);
	\draw (10.5,0.5) circle(0.25);
\fill[color=red] (10.5,1.5) circle(0.25);
	\draw (10.5,1.5) circle(0.25);
\fill[color=red] (10.5,2.5) circle(0.25);
	\draw (10.5,2.5) circle(0.25);
\fill[color=green] (10.5,3.5) circle(0.25);
	\draw (10.5,3.5) circle(0.25);
\fill[color=green] (10.5,4.5) circle(0.25);
\draw (10.5,4.5) circle(0.25);

\draw[dashed] (8.5,2.5)--(10,0.5) ;
\draw[dashed] (8.5,2.5)--(10,1.5) ;
\draw[dashed] (8.5,2.5)--(10,2.5) ;
\draw[very thick] (8.5,2.5)--(10,3.5) ;
\draw[dashed] (8.5,2.5)--(10,4.5) ;

	\draw[very thick] (10,0)--(10.1,0);
	\draw[very thick] (10,1.0)--(10.1,1.0);
	\draw[very thick] (10,2.0)--(10.1,2.0);
	\draw[very thick] (10,3.0)--(10.1,3.0);
	\draw[very thick] (10,4.0)--(10.1,4.0);
	\draw[very thick] (10,5.0)--(10.1,5.0);

%

\end{tikzpicture}

%% file: prefFig.tex
\begin{tikzpicture}[scale = 0.8]

	\draw plot [smooth cycle] coordinates {(0,0) (3,-1) (5,0) (6,2) (3,4) (-1,2)};


	\draw[dotted, very thick] (2,1.5)--(0,0);
	\draw[dotted, very thick] (2,1.5)--(4.8,-0.3);
	\draw[dotted, very thick] (2,1.5)--(2,2.5);
	\draw[dotted, very thick] (2,2.5)--(4,2.3);
	\draw[dotted, very thick] (2,2.5)--(1,3.3);
	\draw[dotted, very thick] (5.7,0.9)--(4,2.3);
	\draw[dotted, very thick] (4.8,3.1)--(4,2.3);


	\draw[fill, color=red] (1,2)--(1.3,2)--(1.3,2.2)--(1,2.2)--cycle; 
	\draw[fill, color=green] (1,2.2)--(1.3,2.2)--(1.3,2.8)--(1,2.8)--cycle; 
	\draw[thick] (1,2)--(1.3,2)--(1.3,2.8)--(1,2.8)--cycle;

	\draw[fill, color=red] (5,2.25)--(5.3,2.25)--(5.3,2.75)--(5,2.75)--cycle; 
	\draw[fill, color=green] (5,2.75)--(5.3,2.75)--(5.3,3.05)--(5,3.05)--cycle; 
	\draw[thick] (5,2.25)--(5.3,2.25)--(5.3,3.05)--(5,3.05)--cycle;

	\draw[fill, color=red] (2.8,0.2)--(3.1,0.2)--(3.1,0.8)--(2.8,0.8)--cycle; 
	\draw[fill, color=green] (2.8,0.8)--(3.1,0.8)--(3.1,1)--(2.8,1)--cycle; 
	\draw[thick] (2.8,0.2)--(3.1,0.2)--(3.1,1)--(2.8,1)--cycle;

	\draw[fill, color=red] (3.0,1.5)--(3.3,1.5)--(3.3,1.9)--(3.0,1.9)--cycle; 
	\draw[fill, color=green] (3.0,1.9)--(3.3,1.9)--(3.3,2.3)--(3.0,2.3)--cycle; 
	\draw[thick] (3.0,1.5)--(3.3,1.5)--(3.3,2.3)--(3.0,2.3)--cycle; 

	\draw[fill, color=red] (3.0,3)--(3.3,3)--(3.3,3.2)--(3.0,3.2)--cycle; 
	\draw[fill, color=green] (3.0,3.2)--(3.3,3.2)--(3.3,3.8)--(3.0,3.8)--cycle; 
	\draw[thick] (3.0,3)--(3.3,3)--(3.3,3.8)--(3.0,3.8)--cycle; 


\fill[color=black] (2,6) circle(0.25);

	\draw[dashed] (2,6)--(1.15,2.8) ;
\draw[dashed] (2,6)--(5.15,3.05) ;
\draw[very thick] (2,6)--(2.8,0.4) ;
\draw[dashed] (2,6)--(3,1.7) ;
\draw[dashed] (2,6)--(3.15,3.8) ;


	\draw[->, >=latex, blue!20!white, line width = 10pt] (3.4,0.5) to[in =-125, out = 0] (8.2,2.2);

	\draw[thick] (10,0)--(11,0)--(11,5)--(10,5)--cycle;

	\fill[color = red!80!white] (10,0)--(11,0)--(11,3.5)--(10,3.5)--cycle;
	\fill[color = green!80!white] (10,3.5)--(11,3.5)--(11,5)--(10,5)--cycle;

\fill[color=black] (8.5,2.5) circle(0.25);

\draw[very thick] (8.5,2.5)--(10,2.3) ;

	\draw[very thick] (10,0)--(10.1,0);
	\draw[very thick] (10,1.0)--(10.1,1.0);
	\draw[very thick] (10,2.0)--(10.1,2.0);
	\draw[very thick] (10,3.0)--(10.1,3.0);
	\draw[very thick] (10,4.0)--(10.1,4.0);
	\draw[very thick] (10,5.0)--(10.1,5.0);

\end{tikzpicture}

%% file: ProofTheorem1.tex
\section{Proof of Theorem \ref{LDP_NoSpatial}}
\label{thm1Sec}
The idea for the proof of Theorem~\ref{LDP_NoSpatial} is to use a Sanov-type result for the transmitter distribution $\nu$ and apply the contraction principle for solutions of the associated differential equation~\eqref{DE_Cont}. 
Then main ingredient in this approach is then the continuity of solutions at user distributions with finite entropy. 
The theory of ODE provides us with conditions under which continuity of solutions w.r.t.~parameters can be inferred. Unfortunately, these results mostly work under Lipschitz assumptions which are stronger than the finite entropy bounds provided in our setting. 
However, we construct a two-step Picard approximation that is tailor-made to provide the right balance between two opposing constraints: It is simple enough to be continuous with respect to the driving measure and at the same sufficiently close to the true solution to satisfy uniform approximation properties. This allows us to employ the LDP tool of exponential approximations \cite[Definition 4.2.14]{dz98} to complete the proof of Theorem~\ref{LDP_NoSpatial}.

\subsection{Markovian structure of $\G^\la$}
In this subsection we show that the solutions to equation \eqref{DE_Discrete}, modeled using the random variable $U_j$, is equal in distribution to the process of frustrated users as defined in \eqref{Busy_Process}.
To make this precise, let $\tilde \G^\la$ denote the unique solution of \eqref{DE_Discrete} for the initial condition $\tilde \G^\la_0=0$.
\begin{proposition}\label{Mark}
The processes $\G^\la$ and $\tilde \G^\la$ have the same distribution. 
\end{proposition}

\begin{proof}
Let us fix a realization of $N^\la=|X^\la|$ many transmitter locations $(x_i)_{1\le i\le N^\la}$ together with their data transmission times $(t_i)_{1\le i\le N^\la}$. After that, the measure-valued process $\G^\la=(\G^\la(x_i))_{1\le i\le N^\la}$ is discrete-time and vector-valued. Note that, this process only depends on the random selection of a relay for each transmitter which is done with probability $1/|Y^\la|$. For the process $\tilde \G^\la$ the randomness comes from the uniform distribution of selection variables $U=(U_i)_{1\le i\le N^\la}$. 

For simplicity we assume the data transmission times to be ordered, then both $\G^\la=\{\G^\la_{t_i}\}$ and $\tilde \G^\la=\{\tilde \G^\la_{t_i}\}$ are time-discrete Markov chains the finite state space $\{0,\la^{-1}\}^{N^\la}$. Moreover, at a given time $t_i$ both, $\G^\la$ and $\tilde \G^\la$ can only change in coordinate $x_i$.
For the transition probability of $\G^\la$ we note that if at time $t_i-$, $k$ relays are busy, then at time $t_i$, $x_i$ becomes frustrated with probability $k/|Y^\la|$. Accordingly, with probability $1-k/|Y^\la|$ at time $t_i$, $x_i$ becomes satisfied and the process stays unchanged.
As for $\tilde \G^\la$ if at time $t_i-$, $k$ relays are busy, then $x_i$ becomes frustrated if $U_{i}\in[0, k/|Y^\la|]$. Hence the transition probabilities coincide.
\end{proof}

\subsection{Existence and uniqueness of solutions}\label{Existence and uniqueness of solutions}
For absolutely continuous measures $\MM_{\ms{ac}}=\{\nu\in\MM:\,  \nu\ll\mu_{\ms T}\}$ the existence of solutions to~\eqref{DE_Cont} is non-trivial and will be dealt with in this section. More precisely, let $\LL=\{f\in [0,1]^{[0,\TF]}:\, f \text{ increasing and }f(0)=0\}$, then for $\nu\in\MM_{\ms{ac}}$ we define the integral operator $\mc{T}_\nu:\LL\to\LL$ where
\begin{align}\label{IntOp}
	\TT_{\nu}: \b\mapsto \big(\int_0^t\nu(\d s,[\b_s,1],W)\big)_{t\in[0,\TF]}.
\end{align}
As will be shown below, $\TT_\nu$ is continuous so that existence of solutions can be established via the Schauder-Tychonoff fixed point theorem. The uniqueness is a consequence of a monotonicity property of $\TT_\nu$.
\begin{proposition}\label{ExistenceUniqueness}
	For all $\nu\in\MM_{\ms{ac}}$, there exists exactly one $\b\in\LL$ such that $\TT_{\nu}(\b)=\b$.
\end{proposition}
Note that for the driving measure $\nu=\Lla$, existence and uniqueness of solutions for~\eqref{DE_Cont} is trivial. In both cases, where $\nu=\Lla$ or $\nu\in\MM_{\ms{ac}}$, using \eqref{DE_Cont_r}, we obtain a measure-valued solution which we denote $\b(\nu)$. This solution is a step-function with step-height $\la^{-1}$ if $\nu=\Lla$ and continuous if $\nu\in\MM_{\ms{ac}}$.

\subsection{Approximation scheme for the solution}
In this subsection, we consider time-discretized two-step Picard approximations which will turn out to be convergent in the supremum norm. 
Additionally, the resulting trajectories exhibit good continuity properties w.r.t.~driving measures. Let us start by providing precise definitions of the approximations.

\medskip
Let $\de>0$, $\nu\in\MM$ and consider the discretization of $[0,\TF]$ into $\TF/\de$ disjoint segments of length $\de$, where we assume $\TF/\de$ to be an integer.
To define the approximation, we proceed recursively and start by putting $\bdde_0(\nu)(\d x)=0$. Once $\bdde_{(n-1)\de}(\nu)(\d x)$ is available, we define the locally constant function
\begin{align*}
	\bddde_{t}(\nu)=\bdde_{(n-1)\de}(\nu)(W)+\nu\big(((n-1)\de,n\de]\times[0,1]\times W\big)
\end{align*}
for $t\in((n-1)\de,n\de]$.
Then,  for $t\in((n-1)\de,n\de]$ we put 
\begin{align*}
	\bdde_{t}(\nu)(\d x)=\bdde_{(n-1)\de}(\nu)(\d x)+\nu\big(((n-1)\de,t],[\bddde_{n\de}(\nu),1],\d x\big).
\end{align*}
We can think of $\bddde(\nu)$ as a one-step Picard iteration of the 
zero function. Similarly, $\bdde(\nu)$ corresponds to the two-step Picard iteration.
The approximating process of frustrated transmitters is now defined as
\begin{align*}
	\gdde_{t}(\nu)(\d x)=\nu([0,t],[0,1],\d x)-\bdde_{t}(\nu)(\d x).
\end{align*}

\medskip
Next we show that $\gdde(\nu)$ approximates $\g(\nu)$ sufficiently well to transfer the LDP from $\gdde(\nu)$ to $\g(\nu)$. More precisely, we want to apply the exponential approximation technique~\cite[Theorem 1.13]{eiSchm} and therefore have to show three conditions. First, we show an exponential approximation property on the Banach space of trajectories of finite signed measures equipped with the supremum norm $\Vert\cdot\Vert$ where the supremum is taken over all times and measurable sets. The statement of this auxiliary result makes use of the empirical measures $\Lla$ introduced below~\eqref{DE_Discrete}.
\begin{proposition}\label{ExpEquiv}
$\gdde(\Lla)$ is an $\Vert\cdot\Vert$-exponentially good approximation of $\g(\Lla)$.
\end{proposition}
Second, we show the uniform approximation property on measures with bounded entropy. 
\begin{proposition}\label{DZApprox}
For all $\a\ge0$,
$\limsup_{\de\downarrow 0}\sup_{\nu:\, h(\nu|\mu_{\ms T})\le\a} \Vert\gdde(\nu)-\g(\nu)\Vert=0$.
\end{proposition}
Third, we show continuity of the approximation w.r.t.~the driving measure in the $\tau$-topology, i.e., the topology generated by evaluations on bounded measurable functions.
\begin{proposition}\label{ContractionDiscr}
Let $t\in[0,\TF]$, $\de>0$ and $A\subset W$ measurable, then at any $\nu\in\MM_{\ms{ac}}$, the evaluations $\MM\to[0,\infty)$ given by $\nu\mapsto\gdde_t(\nu)(A)$ are continuous w.r.t.~the $\tau$-topology.
\end{proposition}
In particular, the Propositions~\ref{DZApprox} and \ref{ContractionDiscr} imply that for all $\a>0$, on $\nu:\, h(\nu|\mu_{\ms T})\le\a$, the map $\nu\mapsto \g(\nu)$ is continuous in the $\tau$-topology.

\medskip
Feeding the exponential approximation machinery with the Sanov-type LDP, we obtain the following multivariate LDP. 

\begin{proposition}\label{ProjLimLDP_1}
Let $0\le t_1\le\cdots\le t_k\le\TF$, then the family of random measures $\{\G_{t_i}^\la(\d x)\}_{i}$ satisfies the LDP in the $\tau$-topology with good rate function $I((\g_{t_i})_{i})=\inf_{\nu\in \MM:\, (\g_{t_i}(\nu))_{i}=(\g_{t_i})_{i}}h(\nu|\mu_{\ms T})$.
\end{proposition}

\subsection{Topological lifting and proof of Theorem~\ref{LDP_NoSpatial}}
In this section we first arrive at the continuous-path LDP using the Dawson-G\"artner theorem \cite[Theorem 4.6.1]{dz98}. We work in the product topology in $\MM(W)^{\itf}$, i.e., the coarsest topology such that the evaluations $\g\mapsto\g_t$ are continuous in the $\tau$-topology. The topological lifting is then done using exponential tightness.

\begin{proposition}\label{ProjLimLDP_2}
The family of measure-valued processes $\G^\la$ satisfies the LDP in the product topology. The good rate function is given by $I(\g)=\inf_{\nu \in \MM:\, \g(\nu)=\g}h(\nu|\mu_{\ms T})$.
\end{proposition}

\begin{proof}[Proof of Theorem~\ref{LDP_NoSpatial}]
First note that Proposition~\ref{ProjLimLDP_2}, by contraction, implies the same LDP where the $\tau$-topology is replaced by the weak topology. Then, we may apply~\cite[Corollary 4.2.6]{dz98} to reduce the proof of Theorem~\ref{LDP_NoSpatial} to the exponential tightness of $\G^\la$ in the Skorohod topology. For this we use a criterion from \cite[Theorem 4.1]{FeKu06} and verify its conditions:
	\begin{enumerate}
		\item $\G^\la_t$ is exponentially tight in $\MM(W)$ for all $t\in[0,\TF]$ and
		\item $\limsup_{\de\downarrow0}\limsup_{\la\uparrow\infty}\la^{-1}\log\P(w'_\de(\G^\la)>\e) = -\infty$.
	\end{enumerate}
Here the modulus of continuity $w'$ is defined as
		$$w'_\de(\G^\la)=\inf_{0=t_0<\cdots<t_{k}=\TF:\, \min_{1\le i\le k}|t_{i-1}-t_{i}|>\de}\max_{1\le i\le k}\sup_{s,t\in[t_{i-1},t_{i})}d_{\rm w}(\G^\la_s,\G^\la_t)$$
and with $F^\e$ denoting the $\e$-halo of a closed set $F\subset W$, 
		$$d_{\rm w}(\nu,\nu')=\inf\{\e>0:\, \nu(F)\le\nu'(F^\e)+\e\text{ and } \nu'(F)\le\nu(F^\e)+\e\text{ for all closed }F\subset W\}$$ is the Prokhorov metric on $\MM(W)$ which makes $(\MM(W),d_{\rm w})$ a Polish space, see \cite[Proposition A2.5.III.]{daleyPPII2008}.  
	
\medskip
As for (1) note that $K_\a=\{\nu\in\MM(W):\, \nu(W)\le\a\}$ is compact in the weak topology for any $\a>0$. Using the Poisson concentration inequality \cite[Chapter 2.2]{lugosi},
\begin{align*}
\P(\G^\la_t\in K_\a^c)=\P(\G^\la_t(W)>\a)\le\P(\Lla(V)>\a)=\P(|X^\la|>\a\la)
\le\exp(-\la h(\a|\mu_{\ms T}(V)),
\end{align*}
where $h(x|y)=x\log(x/y)-x+y$. This shows the exponential tightness at every $t$. 

\medskip
For the exponential bound on the modulus of continuity (2) first note that $d_{\rm w}(\G^\la_s,\G^\la_t)\le \sup_{1\le i\le \TF/\de}\Lla([(i-1)\de,i\de]\times[0,1]\times W)=\r$. Indeed, for all closed $F\subset W$ and $i\de\le s\le t\le(i+1)\de$, it suffices to show that
\begin{align}\label{Ineq1}
\G^\la_t(F)\le\G^\la_s(F^\r)+\r
\end{align}
since the other inequality $\G^\la_s(F)\le\G^\la_t(F^\r)+\r$ is trivially satisfied for all $\r>0$.
We can rewrite \eqref{Ineq1} equivalently as, 
\begin{align*}
&\Lla([s,t]\times[0,1]\times F)\le [B^\la_t(F)-B^\la_s(F)]+[\Lla([0,s]\times[0,1]\times F^\r\sm F)-B^\la_s(F^\r\sm F)]+\r
\end{align*}
where the first two summands on the r.h.s.~are nonnegative. By our definition of $\r$, we arrive at the desired bound. Consequently, using the Poisson concentration inequality again,
\begin{align*}
\limsup_{\la\uparrow\infty}\la^{-1}\log\P(w'_\de(\G^\la)>\e)&\le\limsup_{\la\uparrow\infty}\la^{-1}\log\P(\sup_{1\le i\le \TF/\de}\Lla([(i-1)\de,i\de]\times[0,1]\times W)>\e)\cr
&\le\limsup_{\la\uparrow\infty}\la^{-1}\log\sum_{1\le i\le \TF/\de}\P(\Lla([(i-1)\de,i\de]\times[0,1]\times W)>\e)\cr
&\le\max_{1\le i\le \TF/\de}-h(\e|\mu^{\ms s}_{\ms T}(W)\mu^{\ms t}_{\ms T}([(i-1)\de,i\de]))\cr
&=-h(\e|\mu^{\ms s}_{\ms T}(W)\max_{1\le i\le \TF/\de}\mu^{\ms t}_{\ms T}([(i-1)\de,i\de])).
\end{align*}
Since $\mu^{\ms t}_{\ms T}$ is assumed to be absolutely-continuous w.r.t.~the Lebesgue measure, this tends to minus infinity as $\de$ tends to zero, as required.
\end{proof}

%% file: SupportTheorem1.tex
\section{Proofs of supporting results for Theorem~\ref{LDP_NoSpatial}}
\label{thm1SupSec}
In this section, we provide proofs for the Propositions~\ref{Mark}-\ref{ProjLimLDP_2}. To ease notation we will write in the following $\D_\de(i)=((i-1)\de,i\de]$. Let us start by stating three results  that we will use multiple times in the sequel. 

\begin{lemma}\label{AbsoluteContinuity}
Let $\mathcal B(W)=\{A\subset W: \, A \text{Borel measurable}\}$ then the following holds.
\begin{enumerate}
\item	Let $\nu\in\MM_{\ms{ac}}$, then $$\lim_{\e\downarrow0}\sup_{A\subset \mathcal B(W):\, \mu_{\ms T}(A)<\e}\nu(A)=0.$$
\item Let $\a>0$, then $$\lim_{\e\downarrow0}\sup_{\substack{A\subset \mathcal B(W):\, \mu_{\ms T}(A)<\e \\ \nu\in\MM:\, h(\nu|\mu_{\ms T})<\a}}\nu(A)=0.$$
\item Let $\de>0$ then, for a random variable $N^{\e\la}$ which is Poisson distributed with parameter $\e\la$
$$\lim_{\e\downarrow0}\limsup_{\la\uparrow\infty}\la^{-1}\log\P(N^{\e\la}>\la\de)=-\infty.$$
\end{enumerate}
\end{lemma}
\begin{proof}
	Part 1 rephrases the definition of absolute continuity. Part 2 can be shown using Jensen's inequality. Part 3 is a consequence of the Poisson concentration inequality~\cite[Chapter 2.2]{lugosi}.
\end{proof}

\subsection{Existence and uniqueness of solutions} 
Let us start by asserting continuity of the integral operator.

\begin{lemma}
	\label{tContLem}
	Let $\nu\in\MM_{\ms{ac}}$, then the map $\TT_\nu:\,\LL\to\LL$ is continuous in the product topology.
\end{lemma}

\begin{proof}
	We need to show that the map $\b\mapsto\TT_{\nu}(\b)_t$ is continuous for every $t\in[0,\TF]$. 
	Observe that for any $\b'\in\LL$ and $\de>0$,
	\begin{align*}
		|\mc{T}_{\nu}(\b')_t-\mc{T}_{\nu}(\b)_t|&\le\sum_{i=1}^{\TF/\de}\nu\big(\D_\de(i)\times[\b^{'}_{(i-1)\de}\wedge \b_{(i-1)\de},\b^{'}_{i\de}\vee \b_{i\de}]\times W\big).
	\end{align*}
Further note that for $\mu_{\ms T}$ replacing $\nu$ on the r.h.s., we can further bound from above by
	\begin{align}\label{EST}
\mu^{\ms s}_{\ms T}&(W)\sup_{1\le i\le n}\mu^{\ms t}_{\ms T}(\D_\de(i))\sum_{i=1}^{n}(\b^{'}_{i\de}\vee \b_{i\de}-\b^{'}_{(i-1)\de}\wedge \b_{(i-1)\de})
	\end{align}	
and 
\begin{align*}
\sum_{i=1}^{\TF/\de}(\b^{'}_{i\de}\vee \b_{i\de}-\b^{'}_{(i-1)\de}\wedge \b_{(i-1)\de})
\le 1+2\sum_{i=1}^{\TF/\de}|\b^{'}_{i\de}-\b_{i\de}|.
\end{align*}
Now, using Lemma~\ref{AbsoluteContinuity} part 1, since $\mu^{\ms t}_{\ms T}$ is absolutely continuous w.r.t.~the Lebesgue measure on $[0,\TF]$, for any $\e>0$, there exists $\de'>0$ such that for all $\de'>\de>0$, 
$\sup_{1\le i\le \TF/\de}\mu^{\ms t}_{\ms T}(\D_\de(i))<\e$. Secondly, for any such $\de$, by product-convergence, there exists a neighborhood of $\b$ such that if $\b'$ is in that neighborhood, $\sum_{i=1}^{\TF/\de}|\b^{'}_{i\de}-\b_{i\de}|\le 1/2$. In particular, \eqref{EST} is bounded from above by $2\e\mu^{\ms s}_{\ms T}(W)$ and can be made arbitrarily small. Since $\nu\ll\mu_{\ms T}$ by assumption, using again Lemma~\ref{AbsoluteContinuity} part 1, this transfers to $\nu$ and the proof is finished.
\end{proof}

Using the above continuity, now existence and uniqueness follow from the Schauder-Tychonoff fixed-point theorem and monotonicity.

\begin{proof}[Proof of Proposition~\ref{ExistenceUniqueness}]
Let us start by showing existence. Note that the Schauder-Tychonoff fixed-point theorem, see~\cite[Theorem II.7.1.10]{GrDu03}, implies existence if $\TT_\nu:\,\LL\to\LL$ is continuous and $\LL$ is a compact, convex subset of a locally convex linear topological space. For this first note that $\R^{[0,\TF]}$ equipped with the product topology is a locally convex topological vector space. Further, note that $\LL$ is closed inside the compact subset $[0,1]^{[0,\TF]}$ and thereby compact. Since a convex combination of increasing functions is also increasing, $\LL$ is also convex. 
By Lemma~\ref{tContLem}, the mapping $\b\mapsto\TT_\nu(\b)$ is continuous, which implies existence.

\medskip
As for the uniqueness, we proceed by contradiction, assuming that there exist two solutions $\b,\b'\in\LL$ of~\eqref{DE_Cont} and a point in time $t_1\in[0,\TF]$ satisfying $\b_{t_1}>\b'_{t_1}$. Then, we let $t_0\in[0,t_1)$ denote the last point before $t_1$, where $\b_{t_0}=\b'_{t_0}$. In particular,
	$$\b_{t_1}=\b_{t_0}+\int_{(t_0,t_1]}\nu(\d s,[\b_{s},1])\le\b'_{t_0}+\int_{(t_0,t_1]}\nu(\d s,[\b'_{s},1])=\b'_{t_1},$$
	which gives the desired contradiction.
\end{proof}

\subsection{Exponential approximation property of the approximation scheme}
Let us first derive some results on dominance and closeness of the approximating trajectories w.r.t.~the original process.  We write $\D$ for the symmetric difference between sets.
\begin{lemma}\label{trivCompLem}
	Let $\de > 0$ and $\nu\in\MM_{\ms{ac}}$ or $\nu=\Lla$. Then, $\bdde(\nu)(W)\le\b(\nu)(W)$.
\end{lemma}
\begin{proof}
We will abbreviate $\b(\nu)(W)=\b(\nu)$ and analogously for $\bdde$.
	It suffices to show that $\bdde_t(\nu)\le \b_t(\nu)$ holds for all $t\in\D_\de(k)$ and $k\in\{0,\ldots,\TF/\de\}$. We proceed by induction over $k$, the case $k=0$ being trivial. 
	Suppose that $\bdde_{k\de}(\nu)\le\b_{k\de}(\nu)$. In order to derive a contradiction, we assume that $\b_{t}(\nu)<\bdde_{t}(\nu)$ for some $t\in\D_\de(k+1)$. 
	
If $\nu\in\MM_{\ms{ac}}$, there exists a largest time $t_1\in[k\de,t)$ such that $\b_{t_1}(\nu)=\bdde_{t_1}(\nu)$. 
	 In particular, for every $s\in(t_1,t)$
	 $$\bddde_s(\nu)\ge\bdde_s(\nu)\ge \b_s(\nu)$$
and as required
\begin{align*}
	\bdde_t(\nu)=\bdde_{t_1}(\nu)+\int_{(t_1,t]}\nu(\d s,[\bddde_{s}(\nu),1],W)\le\b_{t_1}+\int_{(t_1,t]}\nu(\d s,[\b_{s}(\nu),1],W)=\b_{t}(\nu).
\end{align*}

If $\nu=\Lla$, there exists a largest time $t_1\in[k\de,t)$ such that $\bdde_{t_1}(\nu)=\b_{t_1}(\nu)$ and $\bdde_{t_2}(\nu)>\b_{t_2}(\nu)$, where $t_2$ is the next transmission time after $t_1$ in $\Lla$. In particular, 
\begin{align*}
	\bdde_{t_2}(\nu)=\bdde_{t_1}(\nu)+\nu(\{t_2\},[\bddde_{t_2}(\nu),1],W)\le\b_{t_1}+\nu(\{t_2\},[\b_{t_2}(\nu),1],W)\le\b_{t_2}(\nu),
\end{align*}
as required.
\end{proof}
The next lemma asserts an approximation property for scalar trajectories, uniform in time.
\begin{lemma}\label{NormApprox}
	Assume $\nu\in\MM_{\ms{ac}}$ or $\nu=\Lla$, then, for all $\de>0$,
	$$\Vert\bdde(\nu)(W) - \b(\nu)(W)\Vert\le 2\sup_{1\le i\le\TF/\de}\nu(\D_\de(i)\times[0,1]\times W).$$
\end{lemma}

\begin{proof}
Again, we abbreviate $\b(\nu)(W)=\b(\nu)$ and analogously for $\bdde$.
Let $\e=\sup_{1\le i\le\TF/\de}\nu(\D_\de(i)\times[0,1]\times W)$, then by Lemma~\ref{trivCompLem} it suffices to show that 
	$$\b_t(\nu) \le \bdde_t(\nu)+ 2\e$$
	holds for all $t \in [0, \TF]$.
	As before, we proceed by induction on the interval $\D_\de(k)$ containing $t$. If 
	$\bdde_{(k-1)\de}(\nu) + \e\ge \b_{(k-1)\de}(\nu)$, then 
	$$\b_t(\nu) \le \b_{(k-1)\de}(\nu) + \nu(((k-1)\de, t] \times [0,1]\times W) \le (\bdde_{(k-1)\de}(\nu) + \e) + \e \le \bdde_{t}(\nu) + 2\e.$$
	Otherwise, if $\bdde_{(k-1)\de}(\nu) + \e\le \b_{(k-1)\de}(\nu)$, then 
	$$\bddde_{k\de}(\nu) = \bdde_{(k-1)\de}(\nu) + \nu(\D_\de(k) \times [0,1]\times W) \le \b_{(k-1)\de}(\nu).$$
	Hence, 
	\begin{align*}
		\b_t(\nu) &= \b_{(k-1)\de}(\nu) + \int_{(k-1)\de}^t \nu(\d s, [\b_{s-}(\nu), 1],W)\cr
		&\le \b_{(k-1)\de}(\nu) +  \nu(((k-1)\de, t] \times [\bddde_{k\de}(\nu), 1],W)=  \b_{(k-1)\de}(\nu) + \bdde_{t}(\nu) - \bdde_{(k-1)\de}(\nu),
	\end{align*}
	so that the assertion follows from the induction hypothesis.
\end{proof}
We denote in the sequel $I^\de_{t}(\nu)=[\b_{t}(\nu)(W)\wedge\bddde_{t}(\nu),\b_{t}(\nu)(W)\vee\bddde_{t}(\nu)]$ and note the following  approximation property for measure-valued trajectories, uniform in time and over measurable sets. For $\nu\in\MM_{\ms{ac}}$ or $\nu=\Lla$ and for all $\de>0$, 
	\begin{align}\label{NormApprox_2}
\Vert\bdde(\nu) - \b(\nu)\Vert\le\int_0^t\nu(\d s, [\b_{s-}(\nu),1]\D[\bddde_{s-}(\nu),1],W)\le \int_0^{\TF}\nu(\d t,I^\de_{t-}(\nu),W).
	\end{align}

	The following result show that in the setting of empirical measures, the inequality~\eqref{NormApprox_2} gives rise to a strong probabilistic bound on the $\Vert \cdot \Vert$-distance.

\begin{lemma}\label{Previsible}
	Let $\e > 0$ be arbitrary. Then, 	the random variable
$\int_0^{\TF}X^\la(\d t,[B^\la_{t-}(W)-\e, B^\la_{t-}(W) + \e], W)$ is stochastically dominated by a Poisson random variable with parameter $2\e\mu_{\ms T}^{\ms s}(W)$.
\end{lemma}

\begin{proof}
	The proof is based  on the previsiblity of the time integral.  More precisely, let $\r\ge 0$ and recall that conditioned on $|X^\la|=n$ the marks $(U_i)_{1\le i\le n}$ are iid. Ordering $X^\la$ by the time-index yields
\begin{align*}
&\P\Big(\int_0^{\TF}X^\la(\d t,[B^\la_{t-}(W)-\e, B^\la_{t-}(W) + \e], W)>\r\big||X^\la|=n\Big)\cr
	&=\E\Big[\P(\one\{|B_{t_{n-1}}(W) -  U_n| \le \e\}+\sum_{i=1}^{n-1}\one\{|B_{t_{i-1}}(W) -  U_i| \le \e\} >\r\Big|(X^\la_{i})_{1\le i\le n-1}\Big)\Big||X^\la|=n\Big]\cr
&\le\P\Big(1_{[0,2\e]}(U_n)+\sum_{i=1}^{n-1}\one\{|B_{t_{i-1}}(W) -  U_i| \le \e\}  >\r\Big||X^\la|=n\Big)\cr
&\le\P\Big(\sum_{i=1}^{n}1_{[0,2\e]}(U_i)>\r\Big||X^\la|=n\Big).
\end{align*}
In particular,
\begin{align*}
	\P\Big(\int_0^{\TF}X^\la(\d t,[B^\la_{t-}(W)-\e, B^\la_{t-}(W) + \e], W) > \r\Big) \le \P\Big(\sum_{i = 1}^{|X^\la|}1_{[0,2\e]}(U_i)>\r\Big)
\end{align*}
and by independent thinning, $\sum_{i = 1}^{|X^\la|}1_{[0,2\e]}(U_i)$ is a Poisson random variable with the desired parameter.
\end{proof}

Note that, by the definition of $\g$ and $\gdde$, we have
$$\gdde(\nu)-\g(\nu)=\bdde(\nu)-\b(\nu).$$
Therefore in the proofs of Proposition~\ref{ExpEquiv}, \ref{DZApprox} and \ref{ContractionDiscr},  $\g$ and $\gdde$ can be replaced by $\b$ and $\bdde$.

\begin{proof}[Proof of Proposition~\ref{ExpEquiv}]
Let $\e > 0$ be arbitrary. By the definition of exponential good approximations \cite[Definition 1.2]{eiSchm}, we need to check that
\begin{align*}
\limsup_{\de\downarrow 0}\limsup_{\la\uparrow0}\la^{-1}\log\P(\Vert B^\la- \bdde(\Lla)\Vert > \e)=-\infty.
\end{align*}
Using the bound~\eqref{NormApprox_2}, we have for all $\e'>0$ the uniform estimate
\begin{equation}\label{Est1}
\begin{split}
&\P(\Vert B^\la- \bdde(\Lla)\Vert > \e)\le \P\Big(\int_0^{\TF}\Lla(\d t,I^\de_{t-}(\Lla), W) > \e\Big)\cr
&\le\P\Big(\int_0^{\TF}\Lla(\d t,[B^\la_{t-}(W)-\e', B^\la_{t-}(W) + \e'], W) > \e\Big) + \P(\sup_{t\in[0,\TF]}|I^\de_t(\Lla)| > \e').
\end{split}
\end{equation}
For the first summand on the r.h.s.~of \eqref{Est1} we can use Lemmas~\ref{Previsible} and ~\ref{AbsoluteContinuity} part 3.
For the second summand on the r.h.s.~of \eqref{Est1}, note that
\begin{equation*}
\begin{split}
\P(\sup_{t\in[0,\TF]}|I^\de_t(\Lla)| > \e')\le \P(\Vert\bdde(\Lla)(W)-\bddde(\Lla)\Vert > \e'/2)+\P(\Vert B^\la(W)-\bdde(\Lla)(W)\Vert  > \e'/2).
\end{split}
\end{equation*}
By definition, respectively by Lemma~\ref{NormApprox}, we have
\begin{align*}
\P(\Vert\bdde(\Lla)(W)-\bddde(\Lla)\Vert > \e'/2)&\le\P(\sup_{1\le i\le\TF/\de}\Lla(\D_\de(i)\times[0,1]\times W) > \e'/2)\cr
\P(\Vert B^\la(W)-\bdde(\Lla)(W)\Vert > \e'/2)&\le\P(\sup_{1\le i\le\TF/\de}\Lla(\D_\de(i)\times[0,1]\times W) > \e'/4).
\end{align*}
Using again Lemma~\ref{AbsoluteContinuity} part 3, the proof is finished.
\end{proof}

\begin{proof}[Proof of Proposition~\ref{DZApprox}]
	 By the bound~\eqref{NormApprox_2} we have the estimate
\begin{align*}
\Vert\b(\nu) - \bdde(\nu)\Vert
\le \int_0^{\TF} \nu(\d t, I^\de_t(\nu),W).
\end{align*}
Moreover, by Lemma~\ref{NormApprox} and the definition of $\bddde$, 
	$$\sup_{t\in[0,\TF]}|I^\de_t(\nu)| \le 2 \sup_{1\le i\le\TF/\de}\nu(\D_\de(i)\times[0,1]\times W).$$
It follows by Lemma~\ref{AbsoluteContinuity} parts 1 and 2 that 
\begin{align*}
	\limsup_{\de\downarrow 0}\sup_{\substack{t\in\itf\\\nu:\, h(\nu|\mu_{\ms T})\le\a}}|I^\de_t(\nu)|=0.
\end{align*}
Consequently, for all $\e>0$ and sufficiently small $\de>0$, 
\begin{align*}
	&\limsup_{\de\downarrow 0}\sup_{\substack{\nu:\, h(\nu|\mu_{\ms T})\le\a}}\int_0^{\TF} \nu(\d t,I^\de_t(\nu),W) \\
	&\quad\le \limsup_{\e\downarrow 0}\sup_{\substack{\nu:\, h(\nu|\mu_{\ms T})\le\a}}\int_0^{\TF} \nu(\d t,[\b_t(\nu)(W)-\e, \b_t(\nu)(W)+\e],W).
\end{align*}
 Another application of Lemma~\ref{AbsoluteContinuity} parts 1 and 2 gives the result.
\end{proof}

\begin{proof}[Proof of Proposition~\ref{ContractionDiscr}]
Assume $\nu'\in\MM$, $\nu\in\MM_{\ms{ac}}$ and consider $|\bdde_t(\nu')(A)-\bdde_t(\nu)(A)|$ for some $t\in[0,\TF]$ and measurable $A\subset W$. 
Then for $n\in\{1,\dots,\TF/\de\}$ such that $t\in\D_\de(n)$ we have the upper bound
\begin{equation*}
\begin{split}
|\bdde_t(\nu')(A)&-\bdde_t(\nu)(A)|\le |\bdde_{(n-1)\de}(\nu')(A)-\bdde_{(n-1)\de}(\nu)(A)|\cr
&+|\nu'(((n-1)\de,t]\times[\bddde_{n\de}(\nu'),1]\times A)-\nu(((n-1)\de,t]\times[\bddde_{n\de}(\nu),1]\times A)|.
\end{split}
\end{equation*}
Using the estimate
\begin{equation*}
\begin{split}
|\bdde_{i\de}(\nu')(A)&-\bdde_{i\de}(\nu)(A)|\le|\bdde_{(i-1)\de}(\nu')(A)-\bdde_{(i-1)\de}(\nu)(A)|\cr
&+|\nu'(\D_\de(i)\times[\bddde_{i\de}(\nu'),1]\times A)-\nu(\D_\de(i)\times[\bddde_{i\de}(\nu),1]\times A)|, 
\end{split}
\end{equation*}
we can further bound $|\bdde_{(n-1)\de}(\nu')(A)-\bdde_{(n-1)\de}(\nu)(A)|$ from above by
\begin{equation*}
\begin{split}
\sum_{i=1}^{\TF/\de}|\nu'(\D_\de(i)\times[\bddde_{i\de}(\nu'),1]\times A)-\nu(\D_\de(i)\times[\bddde_{i\de}(\nu),1]\times A)|.
\end{split}
\end{equation*}
We will suppress the spatial component $A$ in our notation for the rest of the proof. Since the sum is finite, it suffices to consider any $1\le i\le\TF/\de$ and note that the case where $((i-1)\de,i\de]$ is replaced by $((n-1)\de,t]$ works equivalently.
We can further estimate, 
\begin{equation*}
\begin{split}
&|\nu'(\D_\de(i)\times[\bddde_{i\de}(\nu'),1])-\nu(\D_\de(i)\times[\bddde_{i\de}(\nu),1])|\cr
&\le\nu'(\D_\de(i)\times[\bddde_{i\de}(\nu'),1]\D[\bddde_{i\de}(\nu),1])+|(\nu'-\nu)(\D_\de(i)\times[\bddde_{i\de}(\nu),1])|.
\end{split}
\end{equation*}
Now, for $\nu'$ sufficiently close to $\nu$ in the $\tau$-topology, the second summand can be made arbitrarily small.
Let $\e>0$, then it suffices to show that for all $\nu'$ in a neighborhood of $\nu$ we have 
\begin{equation*}
\begin{split}
\nu'(\D_\de(i)\times[\bddde_{i\de}(\nu'),1]\D[\bddde_{i\de}(\nu),1])<\e.
\end{split}
\end{equation*}
For this, note that for all $\e'$, there exists a neighborhood of $\nu$ such that for all $\nu'$ in that neighborhood
\begin{align*}
|{\bddde_{i\de}(\nu')}-{\bddde_{i\de}(\nu)}|\le\sum_{j=1}^{\TF/\de}
|(\nu'-\nu)(\D_\de(j)\times[0,1])|<\e'.
\end{align*}
%
For such $\nu'$, we thus have
\begin{equation}\label{Est2}
\begin{split}
\nu'(\D_\de(i)\times[\bddde_{i\de}(\nu'),1]\D[\bddde_{i\de}(\nu),1])\le\nu'(\D_\de(i)\times [\bddde_{i\de}(\nu)-\e',\bddde_{i\de}(\nu)+\e']).
\end{split}
\end{equation}
 Applying the definition of $\tau$-convergence for the third time, for $\nu'$ sufficiently close to $\nu$, the r.h.s.~of \eqref{Est2} is close to 
\begin{equation*}
\begin{split}
\nu(\D_\de(i)\times [\bddde_{i\de}(\nu)-\e',\bddde_{i\de}(\nu)+\e'])
\end{split}
\end{equation*}
up to an arbitrarily small error.
%
%
%
%
Finally, applying Lemma~\ref{AbsoluteContinuity} part 1, the proof is finished.
\end{proof}

\subsection{Sanov's theorem and proof of Proposition~\ref{ProjLimLDP_1}}
For a sequence of iid random variables, Sanov's theorem in the $\tau$-topology is one of the cornerstones of large deviations theory. Clearly, this result should remain valid when passing from the iid to the Poisson setting. However, as it is not easy to find a reference, we provide a detailed proof along the G\"artner-Ellis type argumentation presented in~\cite[Section 6.2]{dz98}. In our presentation, we focus on the steps where there is a substantial difference between the Poisson and the iid case. For the convenience of the reader, we adapt the notation from~\cite[Section 6.2]{dz98} where possible.
\begin{proposition}\label{SanovDiscrete}
The random measures $\Lla$ satisfy the LDP in the $\tau$-topology with good rate function given by
$$I(\nu)=h(\nu|\mu_{\ms{T}}).$$
Moreover, the levelsets of $I$ are sequentially compact in the $\tau$-topology.
\end{proposition}

\begin{proof}
	The empirical measure $\Lla$ can be considered as a random variable in the space $\mc{X} = B(V)'$, the algebraic dual of the space of all bounded linear functions on $V$. We consider $\mc{X}$ as a vector space endowed with the topology generated by the evaluations $\nu \mapsto \nu(\phi)$, $\phi \in B(V)$. With this topology, the topological dual $\mc{X}^*$ of $\mc X$ is isomorphic to $B(V)$.
	Since the Laplace functional of a Poisson point process is known in closed form, the limiting logarithmic moment generating function of $\Lla$ can be computed explicitly and is given by 
	$$\Lambda(\phi) = \int [\exp(\phi(v)) - 1]\mu_{\ms T}(d v) ,\qquad \phi \in B(V) .$$
	Since for every $\phi_1, \ldots, \phi_n \in B(V)$ the function $(t_1, \ldots, t_n)\mapsto \Lambda(\sum_i t_i \phi_i)$ is everywhere differentiable,~\cite[Corollary 4.6.11]{dz98} implies that $\Lla$ satisfies the LDP with good rate function given by the Legendre dual $\Lambda^*$ of $\L$. 

	It remains to show that $\Lambda^* = h(\cdot | \mu_{\ms T})$. By duality theory~\cite[Lemmas 4.5.8 and 6.2.16]{dz98}, it suffices to show that 
	\begin{equation}\label{Leg}
\begin{split}
\Lambda(\phi) = \sup_{\nu \in \mc{X}}\{ \nu(\phi) - h(\nu | \mu_{\ms T})\}.
\end{split}
\end{equation}
	In order to show that $\Lambda(\phi)\le\sup_{\nu \in \mc{X}}\{ \nu(\phi) - h(\nu | \mu_{\ms T})\}$ let $\nu_\phi$ be the measure with density $e^\phi$ w.r.t.~$\mu_{\ms T}$. Then, a quick computation shows that $\Lambda(\phi)=\nu_\phi(\phi)-h(\nu_\phi | \mu_{\ms T})$. Conversely, the r.h.s.~of \eqref{Leg} is equal to $\sup_{\nu \in \mc{X}^*}\{ \L(\phi) - h(\nu | \nu_\phi)\}$ and 
the non-negativity of the entropy concludes the identification. Sequential compactness of $h(\cdot|\mu_{\ms T})$ follows from ~\cite[Lemma 6.2.16]{dz98}.
\end{proof}

Recall that $\G^\la$ is constructed from the solution of \eqref{DE_Cont_r} with $r=r_\la$. We will sometimes make this dependence explicit by writing
\begin{align*}
	\g(\Lla, r_\la) = \G^{\la}.
\end{align*}
The results collected so far would allow us to derive the LDP similar to the one in Proposition~\ref{ProjLimLDP_1} where $\G^\la$ is replaced by $\g(\Lla, 1)$. In order to conclude we thus need a final result on the asymptotic contribution of $r_\la$. Let us start with the following dominance result, where we write 
$$B^{\la,r}(\d x) = \Lla([0,t], [0,1], \d x) - \g(\Lla, r)(\d x).$$
\begin{lemma}
	\label{relayDomLem}
	If $s\le r$, then 
	$$\tfrac sr (B^{\la,r}(W) -\la^{-1}) \le B^{\la,s}(W) \le B^{\la,r}(W).$$
\end{lemma}
\begin{proof}
	We prove both claims by induction on the arrival time. Hence, we assume that the desired inequalities are valid up to time $t_{i-1}$. Now, suppose that at time $t_i$ we had 
	$$B_{t_i}^{\la,s}(W) > B_{t_i}^{\la,r}(W).$$
	This is only possible if $B_{t_{i-1}}^{\la,s}(W) = B_{t_{i-1}}^{\la,r}(W)$ and for the $i$'th choice variable $U_i$, drawn at time $t_i$, we have
	$U_i \in[s^{-1}B_{t_{i-1}}^{\la,s}(W), r^{-1}B_{t_{i-1}}^{\la,r}(W)]$.
	But since $r \le s$, this is impossible. 
	
	Similarly, assume that at time $t_i$ we had 
	$$\tfrac sr (B_{t_i}^{\la,r}(W) -\la^{-1}) > B_{t_i}^{\la,s}(W).$$
This is only possible if for the $i$'th choice variable $U_i \in[r^{-1}B_{t_{i-1}}^{\la,r}(W), s^{-1}B_{t_{i-1}}^{\la,s}(W)]$,
	so that 
	$ \tfrac{s}{r}B_{t_{i-1}}^{\la,r}(W) \le B_{t_{i-1}}^{\la,s}(W)$.
	But this implies that 
	$$\tfrac sr (B_{t_i}^{\la,r}(W) -\la^{-1}) = \tfrac sr B_{t_{i-1}}^{\la,r} (W)\le B_{t_{i-1}}^{\la,s}(W) = B_{t_{i}}^{\la,s}(W),$$
	yielding the desired contradiction.
\end{proof}

\begin{proposition}\label{RelayAppr}
	The families of measure-valued processes $\G^\la$ and $\g(\Lla,1)$ are $\Vert\cdot\Vert$-exponentially equivalent.
\end{proposition}

\begin{proof}
Let $\e>0$ be arbitrary.	By the definition of exponential equivalence \cite[Definition 4.2.14]{dz98} and the identity $\Vert \g(\Lla, \rla) - \g(\Lla, 1)\Vert =\Vert B^{\la,r_\la}- B^{\la,1}\Vert$,  we need to check that
\begin{align*}
\limsup_{\la\uparrow\infty}\la^{-1}\log\P(\Vert B^{\la,r_\la}- B^{\la,1}\Vert > \e)=-\infty.
\end{align*}
Arguing similarly as in the proof of Proposition~\ref{ExpEquiv}, for any $\r>0$, the following estimate holds, 
\begin{equation}\label{Est4}
\begin{split}
&\P(\Vert B^{\la,r_\la}- B^{\la,1}\Vert > \e)\le \P\Big(\int_0^{\TF}\Lla(\d t,[\tfrac{B_{t-}^{\la,r_\la}(W)}{r_\la},1]\D[B_{t-}^{\la,1}(W),1], W) > \e\Big)\cr
&\le\P\Big(\int_0^{\TF}\Lla(\d t,[B_{t-}^{\la,1}(W)-\r, B_{t-}^{\la,1}(W) + \r], W) > \e\Big) + \P(\Vert B^{\la,1}(W)-\tfrac{B^{\la,r_\la}(W)}{r_\la}\Vert > \r).
\end{split}
\end{equation}
For the first summand on the r.h.s.~of \eqref{Est4} we can use Lemmas~\ref{Previsible} and~\ref{AbsoluteContinuity} part 3. 
For the second summand on the r.h.s.~of \eqref{Est4} we can further estimate
\begin{align*}
\P(\Vert B^{\la,1}(W)-\tfrac{B^{\la,r_\la}(W)}{r_\la}\Vert > \r)\le\P(\Vert B^{\la,1}(W)-B^{\la,r_\la}(W)\Vert > \r/2)+\P(B_{\TF}^{\la,r_\la}(W) > \tfrac{\r r_\la}{2|1-r_\la|}).
\end{align*}
As for the second summand using Lemma~\ref{AbsoluteContinuity} part 3
we arrive at the desired limit. For the first summand, Lemma~\ref{relayDomLem}, implies that 
	\begin{align*}
		\P(\Vert B^{\la, r_\la}(W) - B^{\la, 1}(W)\Vert > \r/2) \le \P(B_{\TF}^{\la, 1 \vee r_\la}(W) > \tfrac{\r/2 - \la^{-1}}{|1-r_\la|}),
	\end{align*}
so that the desired limit follows from Lemma~\ref{AbsoluteContinuity} part 3.
\end{proof}

\begin{proof}[Proof of Proposition~\ref{ProjLimLDP_1}]
We first use Proposition~\ref{RelayAppr} to remove the $\la$-dependence in the relays. To conclude the proof, we then apply the $\tau$-topology version of the exponential approximation machinery~\cite[Theorem 1.13]{eiSchm}. The conditions are satisfied according to the Propositions~\ref{ExpEquiv}, \ref{DZApprox}, \ref{ContractionDiscr} and \ref{SanovDiscrete}.
\end{proof}

\subsection{Dawson-G\" artner for the temporal component}
In order to derive the LDP for the product topology in the time dimension, we use Proposition~\ref{ProjLimLDP_1} and apply the Dawson-G\"artner theorem \cite[Theorem 4.6.1]{dz98}.
\begin{proof}[Proof of Proposition~\ref{ProjLimLDP_2}]
Consider time-discretizations ${\bf t}=\{(t_0,\dots,t_n):\, 0=t_0<t_1<\cdots<t_n=\TF\}$ and the associated projections 
 $p_{\bf t}(\g)=(\g_{t_i})_{t_i\in{\bf t}}$. The family of all such projections $J$ has a partial ordering induced by inclusion in the family of discretizations $\bf t$. 
Using Proposition~\ref{ProjLimLDP_1} and \cite[Theorem 4.6.1]{dz98}, $\G^\la$ satisfies the LDP in the product topology with good rate function given by 
$$\tilde I(\g)=\sup_{{\bf t}\in J}I_{\bf t}(p_{\bf t}(\g))\quad\text{ where }\quad I_{\bf t}((\g_{t_i})_{t_i\in{\bf t}})=\inf_{\nu\in\MM:\, (\g_{t_i}(\nu))_{t_i\in{\bf t}}=(\g_{t_i})_{t_i\in{\bf t}}}h(\nu|\mu_{\ms T}).$$ 
The proof is finished once we show that $\tilde I(\g)=I(\g)$ where $I(\g)=\inf_{\nu\in\MM:\, \g(\nu)=\g}h(\nu|\mu_{\ms T})$.

\medskip
We first prove $I\ge\tilde I$. Let $\nu'\in\MM$ be such that $\g(\nu')=\g$. Then, in particular $\g_{t_i}(\nu'))_{t_i\in{\bf t}}=(\g_{t_i})_{t_i\in{\bf t}}$ for any time-discretization $\bf t$ and
\begin{align*}
\inf_{\nu\in\MM:\, (\g_{t_i}(\nu))_{t_i\in{\bf t}}=(\g_{t_i})_{t_i\in{\bf t}}}h(\nu|\mu_{\ms T})\le h(\nu'|\mu_{\ms T}),
\end{align*}
so that $\tilde I\le I$.

\medskip
For the other direction, $I\le\tilde I$, first assume that $\g$ is discontinuous in the sense that there exists a measurable set $A\subset W$ and some $t_d\in[0,\TF]$ such that there exists a sequence $t_n\to t_d$ with $\lim_{n\to\infty}\g_{t_n}(A)\neq\g_{t_d}(A)$. Then, we show that $\tilde I(\g)=\infty$ which trivially implies the inequality. Indeed, consider the sequence of time-partitions ${\bf{t}}_n=\{0<t_n<t_d<t_{\TF}\}$. Then, there exists a sequence $\nu_n\in\MM$ such that  
\begin{align*}
\tilde I(\g)&\ge \limsup_{n\uparrow\infty}\inf_{\nu\in\MM:\, (\g_{t_i}(\nu))_{t_i\in{\bf t}_n}=(\g_{t_i})_{t_i\in{\bf t}_n}}h(\nu|\mu_{\ms T})\cr
&\ge \limsup_{n\uparrow\infty}\inf_{\nu\in\MM:\, (\g_{t_i}(\nu)(A))_{t_i\in{\bf t}_n}=(\g_{t_i}(A))_{t_i\in{\bf t}_n}}h(\nu|\mu_{\ms T})\ge\limsup_{n\uparrow\infty} h(\nu_n|\mu_{\ms T})-\e
\end{align*}
and $\g_{t}(\nu_n)(A)=\g_{t}(A)$ for all $t\in{\bf t}_n$. Moreover, setting $A_n=(t_n,t_d]\times[0,1]\times A$
then $\lim_{n\uparrow\infty}\mu_{\ms T}(A_n)=0$ since $\mu_{\ms T}^{\ms t}$ is absolutely-continuous w.r.t.~the Lebesgue measure. On the other hand, by assumption, there exists an $\e>0$ such that
\begin{align*}
\e<|\g_{t_n}(A)-\g_{t_d}(A)|=|\g_{t_n}(\nu_n)(A)-\g_{t_d}(\nu_n)(A)|
\end{align*}
for sufficiently large $n$. Hence $\nu_n(A_n)>\e/2$, so that Lemma~\ref{AbsoluteContinuity} part 2 yields $\tilde I(\g)=\infty$.

It remains to consider the setting, where $\g$ is continuous and $\tilde I(\g)<\infty$. Let ${\bf t}_\de$ denote a finite partition of $[0,\TF]$ with mesh size smaller than $\de>0$. Then again, there exists a sequence $\nu_\de\in\MM$ such that  
\begin{align*}
\tilde I(\g)&\ge \limsup_{\de\downarrow0}\inf_{\nu\in\MM:\, (\g_{t_i}(\nu))_{t_i\in{\bf t}_\de}=(\g_{t_i})_{t_i\in{\bf t}_\de}}h(\nu|\mu_{\ms T})\ge \limsup_{\de\downarrow0}h(\nu_\de|\mu_{\ms T})-\e
\end{align*}
and $\g_t(\nu_\de)=\g_t$ for all $t\in {\bf t}_\de$. Since the levelsets of $h(\cdot|\mu_{\ms T})$ are sequentially compact in the $\tau$-topology, there exists a $\tau$-accumulation point $\nu_*$ for $(\nu_\de)_\de$ and by the time-continuity of $\g$ and the continuity of $\nu\mapsto\g(\nu)$ along sequences of measures with uniformly bounded entropy, we have $\g_t(\nu_*)=\g_t$ for all $t\in[0,\TF]$. Moreover, by the lower semicontinuity of $h(\cdot|\mu_{\ms T})$ we have $ \limsup_{\de\downarrow0}h(\nu_\de|\mu_{\ms T})\ge  h(\nu_*|\mu_{\ms T})\ge I(\g)$. This finishes the proof.
\end{proof}

%% file: ProofTheorem2.tex
\section{Proof of Theorem~\ref{LDP_Spatial} }
\label{thm2Sec}
For the proof of Theorem~\ref{LDP_Spatial} we construct spatially approximating processes by replacing $\k$ with a carefully chosen step function. As for the time approximation considered in the proof of Theorem~\ref{LDP_NoSpatial}, the strongly regularizing property of the differential equation allows us to verify that again the approximation is uniformly close and exponentially approximates the original process.  This reveals a striking methodological similarity between the space approximations appearing in the proof of Theorem~\ref{LDP_Spatial} and the time approximations considered in Section~\ref{thm1Sec}.
To implement this program, we first need to overcome the technical obstacle that the step functions still depend on the empirical relay process. In particular, Theorem~\ref{LDP_NoSpatial} cannot yet be applied. A preliminary step is therefore to replace the relay process by its limiting measure and show that the error made is exponentially small.

\subsection{Exponential equivalence w.r.t.~the relay process}
Let $W^\de=\{W_1,\dots,W_k\}$ be a partition of $W$ into cubes of side length $\de$. If $W$ is not a cube itself, then the $W_i$ are defined as the intersection of the smaller cubes with $W$. The idea is to partition the transmitter process into independent processes, each process confined to choose relays in a given spatial discretization. More precisely, recalling \eqref{kappa},
let $Z^{\la,i}(\nu_{\ms R})$ denote the Poisson point process with intensity measure
\begin{align*}
	\mu^i(\nu_{\ms R})(\d s, \d u,\d x)=\k_{\nu_{\ms R}}(W_i|x)\mu_{\ms T}(\d s, \d u, \d x)
\end{align*}
and let $\Lla^i(\nuR)$ be the associated empirical measure. In other words, $Z^{\la,i}(\nuR)$ is the Poisson point process of transmitters choosing a relay in $W_i$. Now, consider an associated augmented empirical measure given by 
$$\LLL_\la^\de(\nuR)=\sum_{i = 1}^k\LLL_\la^{i}(\nuR)$$
where 
$$\LLL_\la^i(\nuR)=L_\la^{i}(\nuR)\otimes1_{W_i}\frac{\nuR}{\nuR(W_i)}.$$
Note that $\LLL_\la^\de(\nuR)(\d t,\d u,\d x, W_i)=L_\la^{i}(\nuR)(\d t,\d u,\d x)$, so that the total mass of transmitters pointing into $W_i$ is preserved. However, within $W_i$ this mass is now distributed according to $\nuR$ conditioned on $W_i$. In particular, the kernel $y \mapsto \LLL_\la^\de(\nuR)_y$ appearing in \eqref{MidSolution} is constant on $W_i$, where it is given by $\nuR(W_i)^{-1}L_\la^{i}(\nuR)$.
Thus, 
\begin{align}
	\label{localiEq}
	\int_{W_i}\g(\LLL_\la^\de(\nuR)_y) \nuR(\d y) = \int_{W_i}\g(\nuR(W_i)^{-1}L_\la^{i}(\nuR)) \nuR(\d y) = \g(L_\la^{i}(\nuR),\nuR(W_i)),
\end{align}
where we recall the more detailed notation $\g(\cdot,\cdot)$ from the paragraph preceding Lemma~\ref{relayDomLem}.
In the proof of Theorem~\ref{LDP_Spatial} this identification is an essential ingredient to establish a connection to the setting of Theorem~\ref{LDP_NoSpatial}.

\medskip
As a first step, we show that it is possible to switch between $\nuR=l_\la$ and $\nuR=\muR$ without changing substantially the approximating process of frustrated transmitters.
\begin{proposition}\label{ExpEquiv_Spatial}
The family of measure-valued processes $\g(\mathfrak \LLL^\de_\la(\lla),\lla)$ is $\Vert\cdot\Vert$-exponentially equivalent to $\g(\mathfrak \LLL^\de_\la(\muR),\muR)$.
\end{proposition}

Next we show that $\g(\mathfrak \LLL^\de_\la(\lla),\lla)$ is an exponentially good approximation to $\G^\la$.
\begin{proposition}\label{ExpEquiv_Spatial_2}
The family of measure-valued processes $\g(\mathfrak \LLL^\de_\la(\lla),\lla)$ is an $\Vert\cdot\Vert$-exponentially good approximations of $\G^\la$.
\end{proposition}
Combining this with a uniform bound on the spatial discretization, we arrive at the following multivariate LDP.
\begin{proposition}\label{ProjLimLDP_3}
Let $0\le t_1\le\cdots\le t_k\le\TF$, then the family of random measures $\{\G_{t_i}^\la(\d x)\}_{i}$ satisfies the LDP in the $\tau$-topology with good rate function $$I((\g_{t_i})_{i})=\inf_{\nn\in \MM':\, (\g_{t_i}(\nn,\muR))_{i}=(\g_{t_i})_{i}}h(\nn|\mu(\muR)).$$
\end{proposition}

\begin{proof}[Proof of Theorem~\ref{LDP_Spatial}]
Using a slight modification of the proof of Proposition~\ref{ProjLimLDP_2}, 
Proposition~\ref{ExpEquiv_Spatial_2} can be lifted to the same LDP w.r.t.~continuous times in the product topology. In order to finally establish the LDP in the Skorohod topology, the exponential tightness arguments presented in the proof of Theorem~\ref{LDP_NoSpatial} should be applied verbatim.
\end{proof}

%% file: SupportTheorem2.tex
\section{Proof of supporting results for Theorem~\ref{LDP_Spatial}}
\label{thm2SupSec}
In the previous section, we announced our plan to prove Theorem~\ref{LDP_Spatial} using exponential approximation techniques. This technique requires us to couple the original process and the approximations in a way such that the probability of a non-negligible deviation decays at super-exponential speed. In the present section, we provide details on the coupling construction and show how it can be used to derive Propositions~\ref{ExpEquiv_Spatial},~\ref{ExpEquiv_Spatial_2} and~\ref{ProjLimLDP_3}.

\subsection{Total variation bounds for frustrated users}
\label{totVarBoundSec}
Since the technique of exponential approximation hinges upon total-variation bounds, it is essential to understand the regularity properties of frustrated users as a function of the input process. The following result shows that the self-regulating property of the defining ODE gives rise to excellent continuity properties of the solutions.  

\medskip
We observe that the construction of the process $\G^\la$ from the Poisson point process $\Zla$ does not make use of the choice components associated with the Poisson points. Indeed, given the knowledge about the precise locations of the chosen relays, there is no longer any uncertainty on the evolution of frustrated transmitters. Hence, more generally, to any finite counting measure $\nu$ on $\itfW \times \Yla$ we can associate a process of frustrated users $\g(\nu)$. 

For instance, let $\Zdla(\nuR)$ denote the Poisson point process $\itfWY$ with intensity measure $\la\mu^\de(\nuR)$ whose density w.r.t.~$\must \otimes \musp \otimes \lla$ is given by
$$\k^\de_{\nuR}(y|x) =  \sum_{i = 1}^k\one\{y\in W_i\} \frac{\k_{\nuR}(W_i|x)}{\lla(W_i)}.$$
Then, $\g(\la^{-1}\Zdla(\nuR))$ coincides in distribution with $\g(\mathfrak \LLL^\de_\la(\nuR),\lla)$. Indeed, having the identity~\eqref{localiEq} at our disposal, we can decompose into the spatial subdomains $W_i$ and then apply Proposition~\ref{Mark} in each of these domains separately.

In the following result, we show that $\g$ is 2-Lipschitz on counting measures.
\begin{lemma}
	\label{contrLem}
	Let $\nu$, $\nu'$ be finite simple counting measures on $\itfW \times \Yla$. Then,
	$$\Vert\g(\nu) - \g(\nu')\Vert \le 2\Vert \nu - \nu'\Vert.$$
\end{lemma}
\begin{proof}
	In the proof, we identify $\nu$ and $\nu'$ with their support and write $\nu^\cup$ and $\nu^\cap$ for their union and intersection, respectively. Then, by monotonicity, 
	$$\g(\nu^\cap) \le \min\{\g(\nu), \g(\nu')\} \le \max\{\g(\nu), \g(\nu')\} \le \g(\nu^\cup).$$
	Hence, it suffices to show $\Vert\g(\nu^\cup) - \g(\nu^\cap)\Vert \le \Vert \nu^\cup - \nu^\cap\Vert$. 
	We show this if $\nu^\cup \setminus \nu^\cap$ consists of a singleton $z_0 = \{(t_0, x_0, y_0)\}$. The general statement is obtained via induction. In fact, we  can describe precisely how the space-time counting measures $\g(\nu^\cup)$ and $\g(\nu^\cap)$ differ from each other. If $y_0$ has already been occupied at time $t_0$, then $\g(\nu^\cup)$ and $\g(\nu^\cap)$ agree apart from an additional atom at $z_0$. On the other hand, if $y_0$ has not already been occupied at time $t_0$, then let $z_1 = (t_1, x_1, y_1)$ denote the first particle after time $t_0$ that points to $y_0$. If such a particle does not exist, we leave $z_1$ undefined. Again, $\g(\nu^\cup)$ and $\g(\nu^\cap)$ agree apart from at most one atom, namely $z_1$ if defined.
\end{proof}

\subsection{Mixed exponential equivalence}
Next, we consider intermediate approximations which partially replace the limiting relay measure by the empirical measure. For this, we introduce the mixed augmented empirical measures
$$\LLLd(\muR, \lla) = \sum_{i = 1}^k \Lla^{i}(\muR)\otimes1_{W_i}\frac{\lla}{\lla(W_i)}.$$

\begin{proposition}\label{ExpEquiv_Spatial_0}
	The families of measure-valued processes $\g( \LLL^\de_\la(\muR,\lla),\lla)$ and $\g( \LLL^\de_\la(\muR),\muR)$ are $\Vert\cdot\Vert$-exponentially equivalent.
\end{proposition}
\begin{proof}
	We use the identification~\eqref{localiEq} to decompose $\g( \LLL^\de_\la(\muR,\lla),\lla)$ and $\g( \LLL^\de_\la(\muR),\muR)$ as 
	$$\sum_{i = 1}^k\g\big( \LLL^\de_\la(\muR,\lla)(\d t, \d u, \d x, W_i), \lla(W_i)\big)\,\text{ and }\,\sum_{i = 1}^k\g\big(\LLL^\de_\la(\muR)(\d t, \d u, \d x, W_i), \muR(W_i)\big),$$ 
	respectively. Thus, it suffices to show the exponential equivalence for fixed $i$. By definition, $\LLL^\de_\la(\muR,\lla)(\d t, \d u, \d x, W_i)$ and $\LLL^\de_\la(\muR)(\d t, \d u, \d x, W_i)$ are both empirical measures associated with a Poisson point processes with intensity measure $\la\k_{\muR}(W_i|x)\mu_{\ms T}(\d s, \d u, \d x)$, so that an application of Lemma~\ref{RelayAppr} concludes the proof.
\end{proof}

\subsection{Coupling construction}
\label{couplDiscSec}
The coupling construction announced in the beginning of this section is based on an expansion of the state space $\itf \times W \times \Yla$ to $V^{*,\la} = \itf \times W \times \Yla \times \R_{\ge0}$. This allows Poisson point processes of various inhomogeneous intensities to be coupled by considering points whose last coordinate lies below a threshold function. More precisely,  let $\Zlas$ denote a  Poisson point process on 
$$V^{*,\la} = \itf \times W \times \Yla \times \R_{\ge0}.$$
with intensity measure $\la\must \otimes \mula\otimes |\cdot|$, where $|\cdot|$ is the Lebesgue measure  and
$$\mula =  \musp \otimes \lla.$$
For any family of measurable functions $f_\la:\,  W \times \Yla \to [0,\infty)$ let 
$$M(f) = \{(x, y, v):\,v\le f(x, y) \}$$
denote the sub-level set of $f$. Then, projecting the intersection of $\Zlas$ with $M(f)$ onto $\itf \times W \times \Yla$ yields a Poisson point process $\Zla(f)$ on $\itf \times W \times \Yla$ whose intensity measure $\la\mula^f$ is characterized by
$$\frac{\d \mula^f}{\d(\must \otimes \mula)} =f.$$
For instance, the processes $\Zla$ and $\Zdla(\nuR)$ can be recovered by choosing the threshold function to be $\k_{\lla}(y|x)$ and $\k^\de_{\nuR}(y|x)$, respectively.

For bounded measurable $f,g:W \times \Yla \to [0,\infty)$ we note that the signed counting measure
$\Zla(f) - \Zla(g)$ can be decomposed as $\Zlap(f,g) - \Zlam(f,g)$, where 
$$\Zlap(f,g) = \{(T_i,  X_i, Y_i, V_i)\in \Zlas:\, g(X_i, Y_i) \le V_i \le f(X_i, Y_i)\}$$
and 
$$\Zlam(f,g) = \{(T_i,  X_i, Y_i, V_i)\in \Zlas:\, f(X_i, Y_i) \le V_i \le g(X_i, Y_i)\}.$$
In particular, the $\Vert\cdot\Vert$ distance between $\Zla(f)$ and $\Zla(g)$ can be represented as
\begin{align}
	\label{tvZLem}
	\Vert \Zla(f) - \Zla(g)\Vert  = \max\{\Zlap(f,g)(V^{*,\la}), \Zlam(f,g)(V^{*,\la})\}.
\end{align}
Thus, for arbitrary $\fla, \gla: W\times \Yla \to [0,\infty)$ the distance $\Vert \Zla(f) - \Zla(g)\Vert$ is stochastically bounded by a Poisson random variable with intensity $$\la|\fla - \gla|_{\mula}=\la\int|\fla(x,y) - \gla(x,y)|\mula(\d x,\d y).$$

\subsection{Proof of Propositions~\ref{ExpEquiv_Spatial} and ~\ref{ExpEquiv_Spatial_2}}
We note that Lemma~\ref{contrLem} and identity~\eqref{tvZLem}  allow us to reduce the proof of Propositions~\ref{ExpEquiv_Spatial}  and~\ref{ExpEquiv_Spatial_2} to an intensity bound.

\begin{corollary}
	\label{expAppCor}
	Let $\{\fla\}_{\la>0}$ and $\{\fla^\de\}_{\de, \la>0}$ denote families of non-negative $L^1(\mula)$ functions satisfying
	$$\limsup_{\de \downarrow 0} \limsup_{\la \uparrow\infty}|\fla - \fla^\de|_{\mula} =0.$$
	Then $\g(\Zla(\fla^\de))$ are $\Vert\cdot\Vert$-exponentially good approximations of $\g(\Zla(\fla))$.
\end{corollary}
\begin{proof}
	Let $\e>0$ be arbitrary. First, by Lemma~\ref{contrLem}, it suffices to provide suitable bounds on $\P(\Vert \Zla(\fla) - \Zla(\fla^\de)\Vert > \e)$. Now, the identification~\eqref{tvZLem} transforms the distance between $\Zla(\fla)$ and $\Zla(\fla^\de)$ into the mass of the coupling Poisson process in domains of vanishing $\mula$-measure. Hence, an application of part 3 of Lemma~\ref{AbsoluteContinuity} concludes the proof.
\end{proof}

Next, we provide an example of an intensity bound that will also be relevant for the identification of the rate function in the following section. For this purpose, we introduce the mixed preference functions
$$\k^\de_{\nuR, \nuR'}(y|x) = \sum_{i = 1}^k \one\{y\in W_i\} \frac{\k_{\nuR}(W_i|x)}{\nuR'(W_i)}$$
and the associated intensity measure $\mu^\de(\nuR, \nuR')$ determined by
$$ \frac{\d \mu^\de(\nuR,\nuR')}{\d(\must \otimes \musp \otimes \nuR')}(t, x, y) = \k^\de_{\nuR, \nuR'}(y|x).$$
\begin{lemma}
	\label{intBoundLem}
	It holds that
	$\lim_{\de \downarrow 0}|\k_{\muR} - \k^\de_{\muR, \muR}|_{{\musp \otimes \muR}} = 0$.
\end{lemma}
\begin{proof}
	Expanding the definitions, we see that the claim is equivalent to proving that 
	$$\lim_{\de\downarrow0}\int_{W^2} \sum_{i = 1}^k\one\{y\in W_i\}\Big| \k_{\muR}(y|x) - \frac{\int_{W_i} \k_{\muR}(y'|x) \muR(\d y')}{\muR(W_i)}\Big| ({{\musp \otimes \muR}})(\d x, \d y) = 0.$$
	By dominated convergence it suffices to show that the integrand tends to zero for $\musp \otimes \muR$-almost every $(x,y)$. But this is a consequence of the Lebesgue density theorem~\cite{rudin}.
\end{proof}

Now, we can prove Propositions~\ref{ExpEquiv_Spatial} and~\ref{ExpEquiv_Spatial_2}.
\begin{proof}[Proof of Propositions~\ref{ExpEquiv_Spatial} and ~\ref{ExpEquiv_Spatial_2}]

	By Proposition~\ref{ExpEquiv_Spatial_0}, Corollary~\ref{expAppCor} and Lemma~\ref{intBoundLem} it suffices to show that 
	\begin{align}
		\label{intBoundEq1}
		\lim_{\la\uparrow\infty}|\k^\de_{\muR, \lla} - \k^\de_{\lla, \lla}|_{\mula} = 0
	\end{align}
	for every $\de >0$, and 
	\begin{align}
		\label{intBoundEq2}
		\lim_{\la \uparrow \infty} |\k_{\lla} - \k^\de_{\lla, \lla}|_{\mula} = |\k_{\muR} - \k^\de_{\muR, \muR}|_{\mulim}.
\end{align}
We begin by considering~\eqref{intBoundEq1}.	First, $|\k^\de_{\muR, \lla} - \k^\de_{\lla, \lla}|_{\mula}$ is given by 
$$\sum_{i = 1}^k \int_{W} |\k_{\muR}(W_i | x) - \k_{\lla}(W_i | x)| \musp(\d x).$$ 
	In particular, by dominated convergence, it suffices to show that the integrand converges to zero for $\must$-almost every $x \in W$. Hence, let $1\le i \le k$ and $x\in W$ be arbitrary. Then, $\k_{\muR}(W_i | x) - \k_{\lla}(W_i | x)$ is given by 
	$$\frac{\int_{W_i} \k(x, y) \muR(\d y)}{\int_W \k(x, y) \muR(\d y)} - \frac{\int_{W_i} \k(x, y) \lla(\d y)}{\int_W \k(x, y) \lla(\d y)}.$$
	Disregarding a $\must$-nullset, we may assume that $\k(x,\cdot)$ is $\muR$-almost everywhere continuous, so that the weak convergence $\lla \to \muR$ implies~\eqref{intBoundEq1}.

	For the proof of~\eqref{intBoundEq2} note that, by dominated convergence, it suffices to show that for $\musp$-almost every $x$, 
	$$\lim_{\la\uparrow\infty}|\k_{\lla}(\cdot | x) - \k^\de_{\lla, \lla}(\cdot | x)|_{\lla} = |\k_{\muR}(\cdot | x) - \k^\de_{\muR, \muR}(\cdot | x)|_{\muR}.$$
	First, as $\muR$ is the weak limit of $\lla$, both $|\k_{\lla}(\cdot | x) - \k_{\muR}(\cdot | x)|_{\lla}$ and $|\k^\de_{\lla, \lla}(\cdot | x) - \k^\de_{\muR, \muR}(\cdot | x)|_{\lla}$ tend to zero as $\la$ tends to infinity. Therefore, it remains to show that 
	$$\lim_{\la\uparrow\infty}|\k_{\muR}(\cdot | x) - \k^\de_{\muR,\muR}(\cdot | x)|_{\lla} = |\k_{\muR}(\cdot | x) - \k^\de_{\muR}(\cdot | x)|_{\muR},$$
	which again is a consequence of the weak convergence of the relay measure.
\end{proof}

\subsection{Coupling construction for absolutely continuous measures}
\label{couplContSec}
It should not come as a surprise that similar to what we have seen in the empirical setting in Section~\ref{couplDiscSec}, couplings play a vital r\^ole in the identification of the rate function. The procedure in the absolutely continuous setting is very similar, but some care has to be taken since the empirical measures $\lla$ need to be replaced by the limiting measure $\muR$ and the unit interval is added to the state space.
More precisely, we consider measures on the space 
$$V^* = \itfiww \times [0, \k_\infty],$$
where $\k_\infty = \sup_{x, y\in W}\k_{\muR}(y | x)$. To simplify notation, we write $\k$ and $\g(\cdot)$ instead of the more verbose $\k_{\muR}$ and $\g(\cdot, \muR)$. 
If $f:\,W^2 \to [0,\k_\infty]$ is a measurable function and $\nns \in \MMs = \MM(V^*)$, then we let $\nns(f)$ denote the measure on $V^*$ that is defined by restriction to the sublevel set $M(f)$ and forgetting the last coordinate.  For instance, we can recover previously introduced intensity measures as 
$$\mu(\muR) = \mu^*(\k)\qquad \text{ and }\qquad \mu^\de(\muR, \muR) = \mu^*(\k^\de),$$
where
$$\mu^* = \must \otimes U \otimes \musp \otimes \muR \otimes |\cdot|.$$
The decisive advantage offered by the couplings is that they allow for an elegant way of expressing total-variation distances. More precisely, for bounded measurable $f,g:\, W^2 \to [0,\k_\infty]$ we note that the signed measure
$\nn^*(f) - \nn^*(g)$ can be decomposed as $\nnsp(f,g) - \nnsm(f,g)$, where 
$$\frac{\d \nnsp(f,g)}{\d \nn^*}(t,u,x,y,v) =  \one\{ g(x,y)\le v \le f(x,y)  \} $$
and 
$$\frac{\d \nnsm(f,g)}{\d \nn^*}(t,u,x,y,v) =  \one\{ f(x,y)\le v \le g(x,y) \}.$$
In particular, 
\begin{align}
	\label{tvDefEq}
	\Vert\nns(f) - \nns(g)\Vert = \max\{\nnsp(f,g)(V^*), \nnsm(f,g)(V^*)\}.
\end{align}

\subsection{Uniform spatial approximation property}
The coupling introduced in the previous section brings us into the setting of~\cite[Theorem 4.2.23]{dz98} where both the desired rate functions and the rate functions of the approximations are of contraction type. For the rate-function approximation to be useful in the exponential-approximation argument, we need to verify that the approximations are uniform on measures with bounded entropy, i.e., on
$$\Ia = \{\nns \in \MMs:\, h(\nns | \mu^*) \le \alpha\}.$$

\begin{lemma}
	\label{uniformLem}
	Let $\alpha>0$ be arbitrary. Then, $\lim_{\de \downarrow 0} \sup_{\substack{\nns \in \Ia}}\Vert\g(\nn^*(\k)) - \g(\nn^*(\k^\de))\Vert = 0$.
\end{lemma}

Before we prove Lemma~\ref{uniformLem}, we explain how it can be used to derive Proposition~\ref{ProjLimLDP_3}.

\begin{proof}[Proof of Proposition~\ref{ProjLimLDP_3}]
	Although Lemma~\ref{uniformLem} is the main ingredient for the exponential approximation~\cite[Theorem 4.2.23]{dz98}, there are still two further steps that remain to be verified. First, we need to check that the contraction-type rate functions
	$$\inf_{\nns \in \MMs:\, (\g_{t_i}(\nns(\k)))_i = (\g_{t_i})_i} h(\nns | \mus)\qquad\text{ and } \inf_{\nn \in \MM':\, (\g_{t_i}(\nn))_i = (\g_{t_i})_i} h(\nn |\mu( \muR))$$
	are identical.
	Second, the continuity of the map $\Ia\to\MM(W)^{\itf}$, $\nns \mapsto \g(\nns(\k^\de))$ needs to be justified. In order to verify the identity of the rate functions, we prove that 
	\begin{align}
\label{rateIdentEq}
	\inf_{\nns \in \MMs:\, \nns(\k) = \nn} h(\nns | \mus) = h(\nn | \mu(\muR)).
\end{align}
Showing that the l.h.s.~is at most as large as the r.h.s.~is achieved by setting 
	$$\nns_0(\d t, \d u, \d x, \d y, \d v) =\k(y|x)^{-1} \one\{M(\k)\} \nn(\d t, \d u, \d x, \d y) \d v + \one\{M(\k)^c\}\mu^*(\d t, \d u, \d x, \d y,\d v).$$
	For the reverse inequality it can be checked by direct computation that 
	$$h(\nns | \mus) = h(\nns | \nn_0^*) + h(\nn | \mu(\muR)),$$
	so that the non-negativity gives~\eqref{rateIdentEq}.

	Second, we show that for fixed $\de>0$ the map $\nns \mapsto \g(\nns(\k^\de))$ is continuous. For this note that  $\g(\nns(\k^\de))$ decomposes as
	$$\g(\nns(\k^\de)) = \sum_{i = 1}^k \g(\nns_i(\k^\de)),$$
	where 
	$$\nns_i(\k^\de)(\d t, \d u, \d x) = \nns(\k^\de)(\d t, \d u, \d x, W_i)$$ 
	is the restriction of $\nns(\k^\de)$ to the points whose $y$-coordinate is in $W_i$. Similarly, put 
	$$\mu_i(\muR)(\d t, \d u, \d x) = \mu^\de(\muR, \muR)(\d t, \d u, \d x, W_i).$$ 
	Since $h(\nns_i(\k^\de) | \mu_i(\muR)) \le \alpha$, we deduce from Proposition~\ref{DZApprox} that $\g$ is continuous at $\nns_i(\k^\de)$. Combining this observation with the continuity of the partial evaluation maps $\nns\mapsto\nns_i(\k^\de)$ concludes the proof.
\end{proof}

Hence, it remains to prove Lemma~\ref{uniformLem}. We recall from equation~\eqref{MidSolution} that the process $\b(\nn)$ is obtained as a spatial mixture of the corresponding localized processes $\b(\nn_y)$ at receiver locations $y \in W$. Therefore, understanding how sensitive the localized processes are w.r.t.~their input measures lays the groundwork for the global setting.

\begin{lemma}
	\label{scaleBetLem}
	Let $\nu, \nu' \in \MM_{\ms{ac}}$, then $\Vert\b(\nu)(W) - \b(\nu')(W)\Vert \le \Vert\nu - \nu'\Vert$.
\end{lemma}
\begin{proof}
	Let $t \in \itf$ be arbitrary. 
	By symmetry, it suffices to derive an upper bound for $\b_t(\nu) - \b_t(\nu')$, where for ease of notation we omit the evaluation on $W$. Now, let $t_0 \in \itf$ be the last point before $t$ such that 
	$$\b_{t_0}(\nu) \le \b_{t_0}(\nu'),$$
	then, 
	$$\b_t(\nu) - \b_t(\nu') = \int_{t_0}^t \nu(\d s,[\b_s(\nu), 1] , W) - \int_{t_0}^t \nu'(\d s,[\b_s(\nu'), 1] , W).$$
	This difference can be split up into 
	$$\int_{t_0}^t \nu(\d s,[\b_s(\nu), 1] , W) - \int_{t_0}^t \nu(\d s,[\b_s(\nu'), 1] , W)$$
	and 
	$$\int_{t_0}^t \nu(\d s,[\b_s(\nu'), 1] , W) - \int_{t_0}^t \nu'(\d s,[\b_s(\nu'), 1] , W),$$
	where we know that the first expression is negative and therefore can be omitted. It remains to study the last expression, which is at most
	$\Vert\nu - \nu'\Vert,$
	as required.
\end{proof}

Now we use Lemma~\ref{scaleBetLem} to complete the derivation of Lemma~\ref{uniformLem}.

\begin{proof}[Proof of Lemma~\ref{uniformLem}]
	First, 
	$$\Vert\nns(\k)([0,t], [0,1], \d x) - \nns(\k^\de)([0,t], [0,1], \d x)\Vert \le \Vert\nns(\k) - \nns(\k^\de)\Vert$$
	so that by Lemma~\ref{AbsoluteContinuity} part 2, Lemma~\ref{intBoundLem} and identity~\eqref{tvDefEq} it remains to prove the statement with $\g$ replaced by $\b$.  	By absolute continuity, we can perform disintegration of the measures $\nns(\k)$ and $\nns(\k^\de)$ with respect to the relay coordinate. That is, 
	$$\nns(\k)(\d t, \d u, \d x, \d y) = \nns_y(\k)(\d t, \d u, \d x)\muR(\d y)$$
	and 
	$$\nns(\k^\de)(\d t, \d u, \d x, \d y) = \nns_y(\k^\de)(\d t, \d u, \d x)\muR(\d y).$$
	Let $t\in[0,\TF]$ and $A\subset W$ measurable. Then, inserting the definition of $\b$ we see that we need to compare
	$$\int_{\itW} \nns_y(\k)(\d s, [\b_{s}(\nns_y(\k))(W), 1] \times A) \muR(\d y)$$
	with 
	$$\int_{\itW} \nns_y(\k^\de)(\d s, [\b_s(\nns_y(\k^\de))(W), 1] \times A) \muR(\d y).$$
	We decompose this task into providing bounds separately for 
	$$\Big|\int_{\itW} (\nns_y(\k)-\nns_y(\k^\de))(\d s, [\b_s(\nns_y(\k^\de))(W), 1] \times A) \muR(\d y)\Big|$$
	and
	$$\int_{\itfW} \nns_y(\k)(\d s, \I(\b_s(\nns_y(\k))(W), \b_s(\nns_y(\k^\de))(W)) \times W) \muR(\d y),$$
	where 
	$\I(a, b) = [a \wedge b, a \vee b]. $
	The first expression is at most $\Vert \nns(\k) - \nns(\k^\de)\Vert$, so that again Lemma~\ref{AbsoluteContinuity} part 2, Lemma~\ref{intBoundLem} and identity~\eqref{tvDefEq}  yield that
	$$\lim_{\de \downarrow 0}\sup_{ \nns \in  \Ia}\Vert\nns(\k) - \nns(\k^\de)\Vert = 0.$$
	By Lemma~\ref{scaleBetLem}, the second expression is bounded above by $\nns(\k)(C_{\nns, \de})$ where
	$$C_{\nns, \de} = \{(t, u, x, y):\, |u-\b_t(\nns_y(\k))(W)| \le \Vert\nns_y(\k) - \nns_y(\k^\de)\Vert\}.$$
	In particular, by Lemma~\ref{AbsoluteContinuity} part 2 it remains to show that 
	$\lim_{\de \downarrow 0}\sup_{\nns \in \Ia} \mu(\muR)(C_{\nns, \de}) = 0.$
	For this, we note that reversing the disintegration of the relay measure gives that
	\begin{align*}
		\mu(\muR)(C_{\nns, \de}) &\le 2\int_{W^2}\Vert\nns_y(\k) - \nns_y(\k^\de)\Vert \k(y|x) (\musp \otimes \muR) (\d x, \d y) \\
		&\le  2\musp(W)\k_\infty \int_W \max\{\nnsp_y(\k, \k^\de) , \nnsm_y(\k, \k^\de)\} \muR(\d y) \\
		&\le 2\musp(W)\k_\infty (\nnsp(\k, \k^\de) + \nnsm(\k, \k^\de)),
	\end{align*}
	so that another invocation of Lemma~\ref{AbsoluteContinuity} part 2 and Lemma~\ref{intBoundLem} concludes the proof.
\end{proof}

%% file: ArxivMainCapLDPMF.bbl
\begin{thebibliography}{10}

\bibitem{caireDynamic}
D.~Bethanabhotla, O.~Y. Bursalioglu, H.~C. Papadopoulos, and G.~Caire.
\newblock Optimal user-cell association for massive {MIMO} wireless networks.
\newblock {\em IEEE Transactions on Wireless Communications}, 15(3):1835--1850,
  2016.

\bibitem{binBalls}
S.~Boucheron, F.~Gamboa, and C.~L{\'e}onard.
\newblock Bins and balls: large deviations of the empirical occupancy process.
\newblock {\em Ann. Appl. Probab.}, 12(2):607--636, 2002.

\bibitem{lugosi}
S.~Boucheron, G.~Lugosi, and P.~Massart.
\newblock {\em Concentration Inequalities}.
\newblock Oxford University Press, Oxford, 2013.

\bibitem{daleyPPII2008}
D.~J. Daley and D.~D. Vere-Jones.
\newblock {\em An Introduction to the Theory of Point Processes {I/II}}.
\newblock Springer, New York, 2005/2008.

\bibitem{dz98}
A.~Dembo and O.~Zeitouni.
\newblock {\em Large Deviations Techniques and Applications}.
\newblock Springer, New York, second edition, 1998.

\bibitem{eiSchm}
P.~Eichelsbacher and U.~Schmock.
\newblock Exponential approximations in completely regular topological spaces
  and extensions of {S}anov's theorem.
\newblock {\em Stochastic Process. Appl.}, 77(2):233--251, 1998.

\bibitem{FeKu06}
J.~Feng and T.~G. Kurtz.
\newblock {\em Large Deviations for Stochastic Processes}.
\newblock American Mathematical Society, Providence, RI, 2006.

\bibitem{gramMel2}
C.~Graham and S.~M{\'e}l{\'e}ard.
\newblock A large deviation principle for a large star-shaped loss network with
  links of capacity one.
\newblock {\em Markov Process. Related Fields}, 3(4):475--492, 1997.

\bibitem{gramMel1}
C.~Graham and S.~M{\'e}l{\'e}ard.
\newblock An upper bound of large deviations for a generalized star-shaped loss
  network.
\newblock {\em Markov Process. Related Fields}, 3(2):199--223, 1997.

\bibitem{GrDu03}
A.~Granas and J.~Dugundji.
\newblock {\em Fixed Point Theory}.
\newblock Springer-Verlag, New York, 2003.

\bibitem{Leon95}
C.~L\'eonard.
\newblock Large deviations for long range interacting particle systems with
  jumps.
\newblock {\em Ann. Inst. H. Poincar\'e Probab. Statist.}, 31(2):289--323,
  1995.

\bibitem{PaRe16}
R.~I.~A. Patterson and M.~D.~R. Renger.
\newblock Dynamical large deviations of countable reaction networks under a
  weak reversibility condition.
\newblock {\em WIAS Preprint No. 2273}, 2016.

\bibitem{rudin}
W.~Rudin.
\newblock {\em Real and Complex Analysis}.
\newblock McGraw-Hill Book Co., New York, third edition, 1987.

\bibitem{baccelliDynamic}
A.~Sankararaman and F.~Baccelli.
\newblock Spatial birth-death wireless networks.
\newblock {\em arXiv preprint arXiv:1604.07884}, 2016.

\bibitem{ShWe95}
A.~Shwartz and A.~Weiss.
\newblock {\em Large Deviations for Performance Analysis}.
\newblock Chapman \& Hall, London, 1995.

\bibitem{dimRates}
A.~Shwartz and A.~Weiss.
\newblock Large deviations with diminishing rates.
\newblock {\em Math. Oper. Res.}, 30(2):281--310, 2005.

\end{thebibliography}
